\documentclass[11pt]{amsart}

\usepackage{fullpage} 
\usepackage{amsmath, amssymb, amsthm}
\usepackage{tikz-cd} 
\usepackage{hyperref}
\hypersetup{
	colorlinks = false,
	linkbordercolor = {white},
	citebordercolor = {white},
}
\usepackage[alphabetic]{amsrefs}
\usepackage[title]{appendix}
\usepackage{geometry}
\numberwithin{equation}{section}

\theoremstyle{definition}
\newtheorem{thm}[equation]{Theorem} 
\theoremstyle{definition}

\newtheorem{prop}[equation]{Proposition} 
\newtheorem{cor}[equation]{Corollary}

\newcommand{\dist}{{\text {dist}}}
\newcommand{\ovl}[1]{\overline{#1}}

\newcommand{\pair}[1]{\langle #1\rangle}
\newcommand{\ppair}[1]{\left\langle #1\right\rangle}
\newcommand{\pr}[1]{\left(#1\right)}
\newcommand{\set}[1]{\left\{#1\right\}}
\newcommand{\abs}[1]{\left|#1\right|}
\newcommand{\bb}[1]{\mathbb{#1}} \newcommand{\td}[1]{\widetilde{#1}}

\def\fnum{equation}
\newtheorem{Thm}[\fnum]{Theorem}

\newtheorem{Lem}[\fnum]{Lemma}

\newtheorem{Pro}[\fnum]{Proposition}

\newcommand{\RR}{\mathbb R}

\newcommand{\D}{\Delta}

\newcommand{\n}{\nabla}
\newcommand{\ep}{\varepsilon}

\newcommand{\dL}{\mathcal L}
\newcommand{\sbst}{\subseteq}
\newcommand{\h}{\textbf H}
\newcommand{\na}{\nabla}
\newcommand{\la}{\langle}
\newcommand{\ra}{\rangle}

\newcommand{\bd}{\partial}

\newcommand{\supp}{\text{supp}}

\newcommand{\N}{\textbf{n}}
\newcommand{\p}{\mathbf p}

\title[MCF with asymptotically conical singularities]
{Closed mean curvature flows with \\asymptotically conical singularities}
\author{Tang-Kai Lee, Xinrui Zhao}

\address{MIT, Dept. of Math., 77 Massachusetts Avenue, Cambridge, MA 02139-4307}
\email{tangkai@mit.edu and xrzhao@mit.edu}

\date{\today}

\begin{document}
	\begin{abstract}
	In this paper, we prove that for any asymptotically conical self-shrinker, there exists an embedded closed hypersurface such that the mean curvature flow starting from it develops a singularity modeled on the given shrinker. 
    The main technique is the Wa\.zewski box argument, used by Stolarski in the proof of the corresponding theorem in the Ricci flow case.
    As a corollary, our construction, combined with the works of Angenent--Ilmanen--Velázquez and Chodosh--Daniels-Holgate--Schulze, implies the existence of fattening level set flows starting from smooth embedded closed hypersurfaces. 
    These provide examples related to a question asked by Evans--Spruck.
	\end{abstract}

\maketitle

\vspace{-1cm}

\section{\bf Introduction}

We consider a family of immersions $\mathbf{x}\colon M^n\times I\to \bb R^{n+k}$ such that 
$\bd_t \mathbf{x} = \mathbf{H}$
where $\mathbf{H}$ is the mean curvature vector of the immersed submanifold $M_t = \mathbf{x}(M,t).$ 
This family of immersions is called a mean curvature flow. 
Understanding the singularity formations in a mean curvature flow is a fundamental question. 
More precisely, given a mean curvature flow $M_t$ developing a spacetime singularity at $(\textbf 0,0)\in\bb R^{n+k}\times\bb R,$ we can consider a sequence of rescaled mean curvature flows 
$M^i_t:= \lambda_i M_{\lambda_i^{-2}t}$ 
using any sequence $\lambda_i\to\infty.$
Based on Huisken's monotonicity~\cite{H90} and Brakke's compactness \cite{B78}, White \cite{W94} and Ilmanen \cite{I95} proved that $M_t^i$ weakly converges to a self-similar limit flow, called a \textit{tangent flow}, after replacing $\lambda_i$'s with a subsequence.
When the limit flow is smooth, it is generated by a \textit{self-shrinker}, a submanifold satisfying $\h=-\frac{x^\perp}{2}.$ 
Uniqueness of tangent flows, a fundamental issue for singularity analysis, has been discussed for certain shrinkers \cite{S14, CM15, CM19, Z20, CS, LSS, S23, LZ, CSSZ}. 
For related references on uniqueness of tangent cones of minimal submanifolds, one can see \cite{AA, S83, H97, S08, CM14}.

In this paper, we focus on singularities modeled on asymptotically conical self-shrinkers of codimension one. 
An \textit{asymptotically conical} self-shrinker is a self-shrinker that is asymptotic to a regular cone at infinity. 
They are especially important in understanding two-dimensional self-shrinkers; cf.~\cite{W16}. 
Moreover, given any cone in $\bb R^{n+k}$, there exists at most one asymptotically conical self-shrinker asymptotic to it even in the case that the codimension is higher than one \cite{W14,Kh}. 
Also, asymptotically conical self-shrinkers as tangent flows are unique, proven by Chodosh-Schulze \cite{CS} in the hypersurface case and the authors \cite{LZ} in the higher codimensional case.

A fundamental question on the singularity formation of mean curvature flows is whether there is a closed hypersurface that develops a singularity modeled on an asymptotically conical shrinker. 
Note that an asymptotically conical shrinker is always non-compact; so is the self-shrinking solution modeled on it.
In general, it is not easy to determine the precise type of singularities developed along a mean curvature flow. 
Thus, although there are some examples of asymptotically conical shrinkers \cite{N14, KKM}, a construction of a closed hypersurface flow developing such a singularity has not been given before, though it is expected. 
In this paper, we prove the following theorem and confirm this expectation.

\begin{Thm}\label{main}
    Given an asymptotically conical shrinker $\Sigma^n\subset \RR^{n+1}$, there exists an embedded closed hypersurface $M\subset \RR^{n+1}$ such that the mean curvature flow $M_t$ starting from $M$ develops a type I singularity at time $1$ at the origin modeled on $\Sigma$.
\end{Thm}

This confirms the belief in the existence of closed mean curvature flows having asymptotically conical tangent flows.
This finding has applications, drawing on previous research. 
For instance, Chodosh--Choi--Mantoulidis--Schulze \cite{CCMS} proved that for a closed mean curvature flow with an asymptotically conical singularity, a generic initial perturbation ensures the perturbed flow avoids the asymptotically conical singularity. 
Their theorem indicates that encountering an asymptotically conical singularity in closed mean curvature flows is unlikely. 
Indeed, it is anticipated that all non-cylindrical singularities are not generic, a conjecture confirmed for surfaces \cite{CCMS, CCS, BK23}.
Another method addressing this issue, based on local dynamics, was explored by Sun--Xue \cite{SX1, SX2}.

Another immediate consequence of Theorem~\ref{main} is noteworthy. 
In \cite{ES}, Evans--Spruck conjectured that any level set flow starting from a smooth hypersurface does not fatten (cf. the final sentence in \cite{ES}). 
Based on the construction by Angenent--Ilmanen--Velázquez \cite{AIV} and the fattening criterion proven by Chodosh--Daniels-Holgate--Schulze \cite{CDS}, Theorem~\ref{main} implies the following corollary. 
It provides a negative answer to the question raised by Evans--Spruck (cf. \cite{De, I94, AIC}) by showing that fattening does not necessarily happen instantaneously.

\begin{cor}\label{cor:fatten}
    There exists a fattening level set flow starting from an $n$-dimensional smooth embedded closed hypersurface for $n\in\set{3,4,5,6}.$
    In particular, the flow remains smooth for a short time before fattening occurs.
\end{cor}

In \cite{W02}, White announced a result joint with Ilmanen about the existence of a compact smooth embedded surface in $\bb R^3$ for which the uniqueness of enhanced varifold solutions of mean curvature evolution fails.
He described the proof ideas in \cite[Section~4]{W02}, which are completely different from our approach to Corollary~\ref{cor:fatten}.
In a recent work \cite{IW}, Ilmanen--White improved the result and showed that for any large enough $g,$ there exists a smooth embedded closed surface in $\bb R^3$ of genus $g$ such that the level set flow starting from it develops an asymptotically conical singularity (with two ends) and fattens after that singular time.
Their work was based on the resolution of the multiplicity one conjecture by Bamler--Kleiner~\cite{BK23}.
Combining their results with Corollary~\ref{cor:fatten}, we now have $n$-dimensional embedded closed fattening examples for $n\in\set{2,3,4,5,6}.$

There was another recent work by Ketover~\cite{K24}, where he used min-max methods to construct a genus $g$ asymptotically conical shrinker with two ends and an entropy estimate.
Furthermore, the level set flow starting from the asymptotic cone of the shrinker fattens when $g$ is large.
As pointed out in \cite{IW}, combining Ketover's result with Theorem~\ref{main} and arguing as in the proof of Corollary~\ref{cor:fatten} leads to other examples of two-dimensional fattening level set flows starting from smooth embedded closed surfaces.

We briefly mention the strategy to prove Theorem~\ref{main}.
The proof of Theorem~\ref{main} follows the outline of the argument of the corresponding construction in the Ricci flow case by Stolarski \cite{Sto}. 
More precisely, we used the Wa\.zewski box argument to find a good perturbation using the eigenfunctions of the stability operator of the Gaussian functional. 
The goal is to make the mean curvature flow starting from a perturbation of a doubling of an asymptotically conical shrinker develop a singularity modeled on the chosen shrinker itself. 
To be more specific, we perturb the doubling by eigenfunctions and consider the mean curvature flow starting from the perturbations. 
By considering the first time when the flow gets outside of the "boxes with good estimates", we can see that the only possibility is that its unstable component hits the boundary. 
By topological arguments, we can prove that there exists a perturbation which always stays in the box, which means that it develops the corresponding singularity.

An obvious difference between our work and~\cite{Sto} is the extrinsic nature of mean curvature flows.
When working with the Ricci flow, Stolarski perturbed the metric tensor to achieve the goal of constructing the desired closed Riemannian manifold.
In the case of mean curvature flows, we perturb a doubling of a shrinker by, as expected, looking at normal graphs on it.
Thus, the role of the perturbing function $h$ in our work is not the role of the perturbing tensor in~\cite{Sto}.
Instead, since metric tensors come from pairings of first derivatives, perturbing tensors should correspond to the gradient of our perturbing functions.
This requires additional work to relate the usual $C^k$ norms with the homogeneous norms on an asymptotically conical shrinker.

Other than a slight difference mentioned above, there are three main difficulties in the case of mean curvature flows. 
The first one is that we have to write the flow as a normal graph over the shrinker to use the tools from analysis, while in the Ricci flow case \cite{Sto}, one only needs to pull back tensors. 
Moreover, we have to guarantee the embeddedness of the modified doubling of an asymptotically conical shrinker and control the curvature bound on the gluing part, while in the Ricci flow case, we can use cutoff functions to modify the metric directly. 
Finally, the shrinker quantity is more nonlinear than the Ricci curvature and hence complicates the arguments.

We remark that there are constructions of mean curvature flows with singularities modeled on minimal cones, which form another important class of singularities.
In~\cite{V94}, Vel\'{a}zquez constructed mean curvature flows with singularities modeled on the Simons cones.
Following his work, more precise information has been studied for flows with such singularities.
See, for example, \cite{GS18, S23V, Liu}.

This paper is organized as follows. 
In Section \ref{sec:1}, we provide preliminary notations and results for asymptotically conical self-shrinkers 
and talk about the evolution equations of related immersions. 
In Section \ref{sec:3}, we give an explicit construction of an embedded doubling of an asymptotically conical shrinker. 
In Sections \ref{sec:4} and \ref{sec:5}, we talk about the interior estimates and constructions of barrier functions to get the preservation and the improvement of $C^2$ bounds. 
In Section \ref{sec:6}, we use the Wa\.zewski box argument to finish the proof of Theorem~\ref{main} and conclude with the proof of  Corollary~\ref{cor:fatten}.

\subsection*{\bf Acknowledgement}
The authors are grateful to Toby Colding and Bill Minicozzi for their constant support and several inspiring conversations. 
They would like to thank Letian Chen, Otis Chodosh, Alec Payne, Max Stolarski, Ao Sun, Lu Wang, and Brian White for valuable discussions and comments, and to Alec Payne for bringing the work of Angenent, Ilmanen, and Vel\'azquez to their attention and suggesting the application regarding fattening.
During the project, Lee was partially supported by NSF Grant DMS 2005345 and Zhao was supported by NSF Grant DMS 1812142, NSF Grant DMS 1811267 and NSF Grant DMS 2104349.

\section{\bf Preliminaries facts and evolution equations}\label{sec:1}

\subsection{Analysis on asymptotically conical shrinkers}

Let $\Sigma$ be an asymptotically conical shrinker in $\bb R^{n+1}.$ 
In this paper, this means that its asymptotic cone 
\begin{align}\label{cone-C}
    \mathcal C = \lim\limits_{t\to 0}\sqrt{-t}\Sigma
\end{align}
is a regular cone, i.e., a cone whose singular set consists of the origin.
This immediately implies a curvature bound
\begin{align}\label{A-decay}
    \abs{\n^j A}
    \le \frac {C_j}{|x|^{j+1}}
\end{align}
for some $C_j = C_j(\Sigma)<\infty.$
One can see \cite{W14} for a proof of this.

The weighted Laplacian operator $\dL$ with respect to the Gaussian functional and the stability operator of the Gaussian functional are given by (cf. \cite{CM12})
\begin{align*}
    \dL u
    &:= \D u - \frac 12 \pair{x^T, \n u}, \text{ and}\\
    Lu
    &:= \dL u + \pr{|A|^2+\frac 12}u,
\end{align*}
where $\D$ is the Laplacian on $\Sigma$ and $A$ is the second fundamental form of $\Sigma.$
It is a self-adjoint operation when acting on the $W^{1,2}$ Gaussian space with respect to the weight $e^{-\frac{|x|^2}{4}}dV_g.$

The corresponding $L^2$ and $H^k$ spaces will be denoted by $L^2_W(\Sigma)$ and $H^k_W(\Sigma).$
As in the Ricci flow case (\cite[Lemma 2.16]{Sto}), the inclusion map given by $H^1_W(\Sigma)\to L^2_W(\Sigma)$ is a compact embedding.
This is rigorously proven in \cite[Appendix~B]{BW17} (cf. \cite{CS,LZ}).
In \cite{CS}, it is also shown that $L\colon H^2_W\pr{\Sigma}\to L^2_W(\Sigma)$ is a bounded operator.
Hence, elliptic theories imply that the spectrum of $L$ behaves as nicely as we expect.

\begin{prop}\label{prop:eigenfunctions}
    There exists an $L^2_W$-orthonormal basis $\{h_i\}_{i\in\bb N}$ of $L^2_W(\Sigma)$ such that for all $i\in\bb N,$ $h_i$ is an eigenfunction of $L$ with eigenvalue $\lambda_i,$ that is, $L h_i = - \lambda_i h_i.$\footnote{
    We use the different sign convention about eigenvalues from the one used in \cite{Sto}.
    }
    The eigenvalues $\lambda_i$'s satisfy $\lambda_1<\lambda_2\le \lambda_3\le\cdots$ tending to $+\infty,$
    and each eigenspace is finite dimensional.
\end{prop}

The next thing we need is a nice decaying property of these eigenfunctions.
Similar to the eigenfunctions of the weighted Lichnerowicz operator, one expects that the $L$-eigenfunctions also decay in a suitable manner.
This part follows from a general result proven by Colding--Minicozzi in \cite{CM21}.
A simpler case for $\dL$-eigenfunctions was dealt with by Colding--Minicozzi in \cite{CM20}.

\begin{thm}
    \label{thm:eigen-decay}
    Let $\Sigma$ be an asymptotically conical shrinker in $\bb R^{n+1}.$
    Given $\delta>0$ and $\lambda\in\bb R,$ there exists $C = C(\Sigma,\delta,\lambda)<\infty$ such that the following holds.
    If $Lh = - \lambda h$ with $\|h\|_{L^2_W}=1,$ then
    \begin{align*}
        |h(x)|\le C|x|^{2\pr{\lambda + \delta+\frac 12}}.
    \end{align*}
\end{thm}

\begin{proof}
    The key property that plays a role here is the rapid curvature decay of $\Sigma.$
    In fact, if $Lh=-\lambda h,$ we have
    $h\dL h
    = h\pr{L-\pr{|A|^2+\frac 12}}h.$
    Given a positive $\delta>0,$ by \eqref{A-decay}, we can find $R>0$ such that $|A|^2\le \delta$ on $\Sigma\setminus B_R.$
    Then, on $\Sigma\setminus B_R,$ we have
    \begin{align*}
        h\dL h
        \ge -\lambda |h|^2 - \pr{\delta+\frac 12} h^2.
    \end{align*}
    Now we apply \cite[Theorem~3.9]{CM21}
    by considering $f=\frac{|x|^2}4,$ $b=|x|=r,$ and $\psi = Ch^2$ in their work and obtain the desired estimate for $h.$
    Note that the $\lambda$-eigenspace is finite-dimensional, so we can find a constant depending on $\lambda$ working for all $\lambda$-eigenfunctions.    
\end{proof}

Finally, we mention some weighted spaces that we will use following \cite{CS}.
Let $\tilde{r}$ be a fixed smooth function such that $\tilde{r}\geq 1$ and $\tilde{r}=|x|$ for $|x|\geq 2$. 
We consider the weighted norms 
 \begin{align*}
     \|f\|_{0;-\gamma}^{\rm hom}&:=\sup_{x\in\Sigma}\tilde{r}(x)^{\gamma}|f(x)|,\text{and}\\
     [f]_{\alpha;-\gamma-\alpha}^{\rm hom}&:=\sup_{x,y\in \Sigma}\frac{1}{\tilde{r}(x)^{-\gamma-\alpha}+\tilde{r}(y)^{-\gamma-\alpha}}\frac{|f(x)-f(y)|}{|x-y|^\alpha}.
\end{align*}
The space $C^{0,\alpha}_{{\rm hom},-\gamma}$ is defined to collect all $f$ such that
\begin{align*}
     \|f\|_{C^{0,\alpha}_{{\rm hom},-\gamma}}
     :=\|f\|_{0;-\gamma}^{\rm hom}+[f]^{\rm hom}_{\alpha;-\gamma-\alpha}<\infty,
\end{align*}
and the space $C^{1,\alpha}_{{\rm hom},-\gamma}$ is defined to collect all $f$ such that
\begin{align*}
     \|f\|_{C^{1,\alpha}_{{\rm hom},-\gamma}}
     :=\sum_{j=0}^1\|(\na_\Sigma)^jf\|_{C^{0,\alpha}_{{\rm hom},-j-\gamma}}<\infty.
\end{align*}

\subsection{\bf Evolution equation of immersions}
\label{evolu}
Next, we are going to view hypersurfaces as immersions into a Euclidean space.
Given a hypersurface $\Sigma\subset \bb R^{n+1}$, if we view $(\Sigma^n,\,g)$ as a Riemannian manifold, where $g = g_\Sigma$ is the metric induced from the Euclidean one, there is a canonical map 
$F=(f_1,f_2,\cdots,f_{n+1})\colon \Sigma\to \RR^{n+1}$ which corresponds to the embedding. 
Then we have $g(X,Y)=\la X(F),\,Y(F)\ra.$
When $\Sigma$ is a self-shrinker, the shrinker equation now reads
\begin{align}\label{shrinker}
    \Delta^gF=-\frac{F^\perp}{2}.
\end{align}
We claim that the right hand side of \eqref{shrinker} can be expressed in terms of $\na^g f_i$. 
In fact, in normal coordinates of $(\Sigma,g)$ around a point $x$, we have $(\na f_1,\,\na f_2,\,\cdots, \na f_{n+1} )$ is an $n\times (n+1)$-matrix. 
From the definition, we know that the normal direction is perpendicular to the row space of this matrix. 
From linear algebra theories, we have that the normal vector can be expressed by
\begin{align*}
    \N_F = 
    \frac{(a_1,\cdots,a_{n+1})}{\sqrt{\sum_{j=1}^{n+1}a_j^2}}
    \text{ where }a_i=(-1)^{i+1}\det(\na f_1,\cdots,\na f_{i-1},\na f_{i+1},\cdots \na f_{n+1}).
\end{align*}
We then have
$F^\perp=\la F,\, \N_F\ra \N_F$
can be expressed in terms of $f_i$'s and $\na f_i$'s.

We consider another family of immersions 
$F_0:\Sigma\times I\to \RR^{n+1}$. 
It induces another family of metrics 
\begin{align*}
    g_0(t)(X,Y)=\la X(F_0(\cdot,t)),\,Y(F_0(\cdot,t))\ra.
\end{align*}
We aim at constructing a family of diffeomorphisms $\varphi_t$ of $\Sigma$ and functions $h$ such that \begin{align*}
    \frac{1}{\sqrt{1-t}}F_0(\varphi_t(x),t)
    = F(x) + h(x,t)\N_F(x)
    = F_1(x,t),
\end{align*}
so $\varphi_t(x)$ satisfies 
\begin{align*}
      &(\partial_tF_0)(\varphi_t(x),t)+dF_0(\partial_t\varphi_t(x))(\varphi_t(x),t)
      = \sqrt{1-t}\,\partial_t{h}(x,t)\N_F(x)
      -\frac{1}{2\sqrt{1-t}}\pr{F(x)+{h}(x,t)\N_F(x)}.
\end{align*}
Thus, we have 
\begin{align}
     &(\partial_tF_0)(\varphi_t(x),t)+dF_0(\partial_t\varphi_t(x))(\varphi_t(x),t)+\frac{1}{2\sqrt{1-t}} F(x)\notag
     =\pr{\sqrt{1-t}\,\partial_t{h}(x,t)-\frac{1}{2\sqrt{1-t}}{h}(x,t)}\N_F(x).\notag
\end{align}
As we assume that locally $\frac{1}{\sqrt{1-t}}F_1$ can be written as a graph over $F$, we have 
$
    \la\N_{F_0}(\varphi_t(x)),\N_F(x)\ra\neq 0.
$
We write 
\begin{align}
    \N_F(x)
    = \la\N_F(x),\N_{F_0}(\varphi_t(x))\ra\N_{F_0}(\varphi_t(x))
    +(\N_F(x) - \la\N_F(x),\N_{F_0}(\varphi_t(x))\ra\N_{F_0}(\varphi_t(x))),
\end{align}
and we can take 
\begin{align*}
  dF_0\pr{\partial_t\varphi_t(x)}
  =& \frac{\la (\partial_tF_0)(\varphi_t(x),t)+\frac{1}{2\sqrt{1-t}} F(x),\N_{F_0}(\varphi_t(x))\ra}{\la\N_F(x),\N_{F_0}(\varphi_t(x))\ra}(\N_F(x)-\la\N_F(x),\N_{F_0}(\varphi_t(x))\ra\N_{F_0}(\varphi_t(x)))\notag\\ 
  &- ((\partial_tF_0)(\varphi_t(x),t)+\frac{1}{2\sqrt{1-t}} F(x)-\la (\partial_tF_0)(\varphi_t(x),t)+\frac{1}{2\sqrt{1-t}} F(x),\N_{F_0}(\varphi_t(x))\ra\N_{F_0}(\varphi_t(x)))\\
  =&\frac{\la (\partial_tF_0)(\varphi_t(x),t)+\frac{1}{2\sqrt{1-t}} F(x),\N_{F_0}(\varphi_t(x))\ra}{\la\N_F(x),\N_{F_0}(\varphi_t(x))\ra}\N_F(x)-((\partial_tF_0)(\varphi_t(x),t)+\frac{1}{2\sqrt{1-t}} F(x)).
\end{align*}
In particular, we have 
\begin{align*}
    |dF_0(\partial_t\varphi_t(x))|
    \leq \frac{\abs{(\partial_tF_0)(\varphi_t(x),t)+\frac{1}{2\sqrt{1-t}} F(x)}}
    {|\la\N_F(x),\N_{F_0}(\varphi_t(x))\ra|}.
\end{align*}
Therefore, we have
\begin{align}
    \partial_t{h}(x,t)=\frac{1}{2(1-t)}{h}(x,t)+\frac{\frac{1}{\sqrt{1-t}}\la(\partial_tF_0)(\varphi_t(x),t),\N_{F_0}(\varphi_t(x))\ra+\frac{1}{2(1-t)}\la F(x),\N_{F_0}(\varphi_t(x))\ra}{\la\N_F(x),\N_{F_0}(\varphi_t(x))\ra}.
\end{align}
Setting $e^\tau=\frac{1}{1-t},$
we get
$
    \partial_t=\frac{\partial \tau}{\partial t}\partial_\tau=\frac{1}{1-t}\partial_\tau,
$
and thus,
\begin{align*}
    \partial_\tau{h}(x,t)
    =&\frac{1}{2}{h}(x,t)+\frac{{\sqrt{1-t}}\la (\partial_tF_0)(\varphi_t(x),t),\N_{F_0}(\varphi_t(x))\ra+\frac{1}{2}\la F(x),\N_{F_0}(\varphi_t(x))\ra}{\la\N_F(x),\N_{F_0}(\varphi_t(x))\ra}\\ 
    =& \frac{\la \Delta^{g_1}F_1,\N_{F_1}\ra(x,t)}{\la\N_F(x),\N_{F_1}(x)\ra}+\frac{{\sqrt{1-t}}\la (\partial_tF_0-\Delta^{g_0}F_0)(\varphi_t(x),t),\N_{F_0}(\varphi_t(x))\ra}{\la\N_F(x),\N_{F_1}(x)\ra}+\frac{\frac{1}{2}\la F_1(x),\N_{F_1}(x)\ra}{\la\N_F(x),\N_{F_1}(x)\ra}.
\end{align*}
If we denote 
\begin{align*}
    F_1=( (f_1)_1,(f_1)_2,\cdots,(f_1)_{n+1})
    \text{ and }h\,\N_F=( h_1,h_2,\cdots,h_{n+1}),
\end{align*}
we have
\begin{align*}
    \N_{F_1}
    = \frac{(\widetilde{a_1},\cdots,\widetilde{a_{n+1}})}{\sqrt{\sum_{j=1}^{n+1}\widetilde{a_j^2}}}\text{ where } 
    \td{a_i}=(-1)^{i+1}\det\pr{\na (f_1+h_1),\cdots,\widehat{\na (f_{i}+h_{i})},\cdots, \na (f_{n+1}+h_{n+1})}.
\end{align*}
In addition, we denote 
\begin{align*}
    b_i
    &=\td{a_i} - a_i\\
    & =(-1)^{i+1}\pr{\det\pr{\na (f_1+h_1),\cdots,\widehat{\na (f_{i}+h_{i})},\cdots \na (f_{n+1}+h_{n+1})}
    - \det\pr{\na f_1,\cdots,\widehat{\na f_i},\cdots \na f_{n+1}}}.
\end{align*}
Then, we derive
\begin{align*}
\td \N
:= \N_{F_1}-\N_F
=\pr{\frac{a_i\pr{\sum_{j=1}^{n+1}-2a_jb_j-b_j^2} + b_i\sqrt{\sum_{j=1}^{n+1}a_j^2}\pr{\sqrt{\sum_{j=1}^{n+1}(a_j+b_j)^2}+\sqrt{\sum_{j=1}^{n+1}a_j^2}}}
{\sqrt{\sum_{j=1}^{n+1}(a_j+b_j)^2}\sqrt{\sum_{j=1}^{n+1}a_j^2}
\pr{\sqrt{\sum_{j=1}^{n+1}(a_j+b_j)^2}+\sqrt{\sum_{j=1}^{n+1}a_j^2}}}
}_{i=1,2,\cdots,n+1}.
\end{align*}

In local coordinates, we have
\begin{align*}
     &\frac{\ppair{\Delta^{g_1}F_1+\frac{F_1^\perp}{2},\N_{F_1}}} {\la \N_F,\N_{F_1} \ra}\N_F
     =\frac{\ppair{\Delta^{g_1}F_1+\frac{F_1^\perp}{2},\N_F}} {\la \N_F,\N_{F_1}\ra^2}\N_F\\
     =&\ppair{\Delta^{g}( h)+\frac{F_1^\perp-F^\perp}{2}+(\Delta^{g_1}-\Delta^g)F_1,\N_F} \frac{\N_F}{\la \N_F,\N_{F_1}\ra^2}\\ 
     =&\ppair{\Delta^{g}( h\,\N_F)
     -\frac{1}{2}\pr{\la F,\N_F\ra\N_F-\la F_1,\N_{F_1}\ra\N_{F_1}}
     +\pr{g_1^{ab}-g^{ab}} \partial_a\partial_bF_1
     -\pr{g_1^{ab}\Gamma^{g_1,c}_{ab}-g^{ab}\Gamma^{g,c}_{ab}} \partial_c F_1,\N_F} \frac{\N_F}{\la \N_F,\N_{F_1}\ra^2}\\
     =&\left\langle
     \Delta^{g}( h\,\N_F) + \frac{1}{2}\la ( h\,\N_F),\N_F\ra \N_F
     -\frac{1}{2}(\la F_1,\N_F\ra\N_F-\la F_1,\N_F
     +\td \N\ra(\N_F+\td \N)) + \td{( h \N_F)}^{ab} \partial_a\partial_b F_1 \right.\\  
     &\left. -\td{( h\,\N_F)}^{ab}\Gamma^{g,c}_{ab}\partial_c F_1-g_1^{ab}(\Gamma^{g_1,c}_{ab}
     -\Gamma^{g,c}_{ab})\partial_c F_1,\N_F\right\rangle
     \frac{\N_F}{\la \N_F,\N_{F_1}\ra^2}\\
     =& \left\langle\Delta^{g}( h\,\N_F)+\frac{1}{2}\la ( h\,\N_F),\N_F\ra \N_F+\frac{1}{2}(\la F+( h\,\N_F),\td \N\ra(\N_F+\td \N)+\la F+( h\,\N_F),\N_F\ra\td \N)+\tilde{h}^{ab}\partial_a\partial_bF_1\right.\\ 
     &-\tilde{h}^{ab}\Gamma^{g,c}_{ab}\partial_c F_1-\frac{1}{2}((\tilde{h}^{ab}+g^{ab})(g^{cd}+\tilde{h}^{cd})(\hat{h}_{da,b}+\hat{h}_{db,a}-\hat{h}_{ab,d})-\tilde{h}^{ab}\tilde{h}^{cd}(g_{da,b}+g_{db,a}-g_{ab,d}))\partial_c F_1\\
     &\left.-\frac{1}{2}(g^{ab}\tilde{h}^{cd}(g_{da,b}+g_{db,a}-g_{ab,d}))\partial_c F_1,\N_F\right\rangle
     \frac{\N_F}{\la \N_F,\N_{F_1}\ra^2}.
\end{align*}
Thus, we have
\begin{align*}
     &\frac{\partial ( h\,\N_F)}{\partial \tau}\\
      =&\left\langle\Delta^{g}( h\,\N_F) 
     + \frac{1}{2}h\,\N_F 
     + \frac{1}{2}(\la F, \td\N\ra \N_F
     + \la F,\N_F\ra\td \N)+\tilde{h}^{ab}\nabla^g_a\nabla^g_b F \right.\\  
     &- \frac{1}{2}\pr{g^{ab}g^{cd}
     \pr{\hat{h}_{da,b}+\hat{h}_{db,a}-\hat{h}_{ab,d}} + g^{ab}\tilde{h}^{cd}(g_{da,b}+g_{db,a}-g_{ab,d})}
     \partial_c F+\frac{1}{2}(\la h\,\N_F,\td \N\ra(\N_F+\td \N)+\la F,\td \N\ra\td \N+h\td \N)\\   
     &+ \tilde{h}^{ab} \nabla^g_a\nabla^g_b (h\,\N_F) -\frac{1}{2} \pr{g^{ab}g^{cd}
     \pr{\hat{h}_{da,b}+\hat{h}_{db,a}-\hat{h}_{ab,d}} + g^{ab}\tilde{h}^{cd} \pr{g_{da,b}+g_{db,a}-g_{ab,d}}}\partial_c (h\,\N_F)\\ 
     &-\left. \frac{1}{2}((\tilde{h}^{ab}g^{cd}+g^{ab}\tilde{h}^{cd})(\hat{h}_{da,b}+\hat{h}_{db,a}-\hat{h}_{ab,d})-\tilde{h}^{ab}\tilde{h}^{cd}(g_{da,b}+g_{db,a}-g_{ab,d}))(\partial_c F+\partial_c (h\,\N_F))),\N_F\right\rangle
     \frac{\N_F}{\la \N_F,\N_{F_1}\ra^2}\\ 
     &+\frac{{\sqrt{1-t}}\la (\partial_tF_0-\Delta^{g_0}F_0)(\varphi_t(x),t),\N_{F_0}(\varphi_t(x))\ra}{\la\N_F(x),\N_{F_1}(x)\ra}\N_F\\ 
     =& \left\langle\Delta^{g}h\cdot\N_F-h|A|^2\N_F+\frac{1}{2}h\,\N_F+\frac{1}{2}(\la F,\td \N\ra\N_F+\la F,\N_F\ra\td \N)+\tilde{h}^{ab}\nabla^g_a\nabla^g_b F \right.\\
     &+\frac{1}{2}(\la h\,\N_F,\td \N\ra(\N_F+\td \N)+\la F,\td \N\ra\td \N+h\td \N)+\tilde{h}^{ab}\nabla^g_a\nabla^g_b (h\,\N_F)\\ 
     &-\frac{1}{2}\pr{g^{ab}g^{cd}(\hat{h}_{da,b}+\hat{h}_{db,a}-\hat{h}_{ab,d})+g^{ab}\tilde{h}^{cd}(g_{da,b}+g_{db,a}-g_{ab,d})}\partial_c h\cdot\N_F\\  
     &-\left.\frac{1}{2} \pr{(\tilde{h}^{ab}g^{cd}+g^{ab}\tilde{h}^{cd})(\hat{h}_{da,b}+\hat{h}_{db,a}-\hat{h}_{ab,d})-\tilde{h}^{ab}\tilde{h}^{cd}(g_{da,b}+g_{db,a}-g_{ab,d})} \partial_c h\cdot\N_F,\N_F\right\rangle
     \frac{\N_F}{(1+\la \N_F,\td \N\ra)^2}\\ 
     &+\frac{{\sqrt{1-t}}\la (\partial_tF_0-\Delta^{g_0}F_0)(\varphi_t(x),t),\N_{F_0}(\varphi_t(x))\ra}{\la\N_F(x),\N_{F_1}(x)\ra}
     \N_F,
\end{align*}
where $\tilde{h}^{ab}=g^{ab}_1-g^{ab}$, $\hat{h}_{ab}=(g_{1})_{ab}-g_{ab}$.

As in the beginning, we denote the second fundamental form on the image of $F$ as $A_{ij}$. 
If we work in normal coordinates around a point, we have 
\begin{align}\label{nota}
   \td \N&=-\nabla^g{h}+C_1(F,\nabla F,\na^2F,h,\nabla h)*\nabla (h\,\N_F)*\nabla (h\,\N_F),\\
    \notag \hat{h}_{ab}&=\la\partial_aF,\partial_b(h\,\N_F)\ra+\la\partial_a(h\,\N_F),\partial_bF\ra+\la\partial_a(h\,\N_F),\partial_b(h\,\N_F)\ra=2hA_{ab}+\la\partial_ah,\partial_bh\ra+h^2A_{ia}A_{bj}g^{ij}, \text{and}\\ \notag 
    \tilde{h}^{ab}&= -2hA^{ab}+C_{2,ab}(F,\nabla F,\na^2F,h,\nabla h)*\nabla h*\na h+C_{3,ab}(F,\nabla F,\na^2F,h,\nabla h)*hA,
\end{align}
where 
$$\|C_1(F,\nabla F,\na^2F,h,\nabla h)\|_{C^0}+\sum_{a,b}\|C_{2,ab}(F,\nabla F,\na^2F,h,\nabla h)\|_{C^0}+\|C_{3,ab}(F,\nabla F,\na^2F,h,\nabla h)\|_{C^0}\leq C_0\|h\|_{C^{1,\alpha}_{\rm hom,-1}}.$$
Also, note that $\Delta^g \N_F=-|A|^2\N_F+\nabla^gH$ and $A_{ij,k}=A_{ik,j}$ in $\RR^{n+1}$. 
Here, we use the notation $\h = \Delta^g F$ and 
$H = g^{ij}A_{ij} = -\ppair{\h,\N_F},$
so in normal coordinates, we have 
\begin{align*}
    &\ppair{\Delta^{g}h\cdot\N_F-h|A|^2\N_F+\frac{1}{2}h \N_F+\frac{1}{2}(\la F,\td \N\ra\N_F+\la F,\N_F\ra\td \N)+\tilde{h}^{ab}\nabla^g_a\nabla^g_b F,\N_F}\\
    =& \left\langle(\Delta^gh-|A|^2h)\N_F+\frac{1}{2}h\,\N_F+2h|A|^2\N_F+\frac{1}{2}(\la F,-\nabla^g{h}\ra\N_F+\la F,\N_F \ra(-\nabla^g{h}))\right.\\ 
    &\left. + C_1'(F,\na F,h,\na h)*\na (h\,\N_F)+C_2'(F,\na F,h,\na h) h\,\N_F,\N_F\right\rangle\\ 
    =& \pr{\Delta^gh+|A|^2h-\frac{1}{2}\la F,\nabla^g h\ra+\frac{1}{2}h}
    +C'_1(F,\na F,h,\na h)*\na (h\,\N_F)+C_2'(F,\na F,h,\na h) h.
\end{align*}

We recall the stability operator $L$ for the Gaussian function mentioned before, written using the embedding $F,$ 
$$Lf=\Delta^gf+|A|^2f-\frac{1}{2}\la F,\nabla^g f\ra+\frac{1}{2}f.$$ 
Combining all the calculations above, we can write the evolution equation of $h$ as
\begin{align}\label{h-equ}
    \frac{\partial h}{\partial \tau}=&Lh+\frac{{\sqrt{1-t}}\la (\partial_tF_0-\Delta^{g_0}F_0)(\varphi_t(x),t),\N_{F_0}(\varphi_t(x))\ra}{\la\N_F(x),\N_{F_1}(x)\ra}
    \\&+C_3(F,\na F,h,\na h)* h+C_4(F,\na F,h,\na h)*\na h+C_5(F,\na F,h,\na h)*\na^2 h
    \notag\\ 
    =& :Lh+\mathcal{Q}_1+\mathcal{Q}_2,\notag 
\end{align}
where
\begin{align*}
    \mathcal{Q}_1
    =&{\sqrt{1-t}}\ppair{(\partial_tF_0-\Delta^{g_0}F_0)(\varphi_t(x),t),\N_{F_0}(\varphi_t(x))}, 
    \text{ and}\\ 
    \mathcal{Q}_2=&C_3(F,\na F,h,\na h)*h+C_4(F,\na F,h,\na h)*\na h+C_5(F,\na F,h,\na h)*\na^2 h \notag
    \\&+{\sqrt{1-t}}
    \ppair{(\partial_tF_0-\Delta^{g_0}F_0)(\varphi_t(x),t),\N_{F_0}(\varphi_t(x))}
    \pr{\frac{1}{{\la\N_F(x),\N_{F_1}(x)\ra}}-1}.
\end{align*}
We remark that in normal coordinates, the second order term in $\mathcal{Q}_2$ corresponds to 
\begin{align}\label{Q_2second}
     &(\tilde{h}^{ab}\partial_a\partial_bh-g^{ab}g^{cd}\partial_a\partial_bh\,\partial_dh\,\partial_c h-(\tilde{h}^{ab}g^{cd}+g^{ab}\tilde{h}^{cd})\partial_a\partial_bh\,\partial_dh\,\partial_c h)\frac{\N_F}{(1+\la \N_F,\td \N\ra)^2}\\
        \notag 
    =&\frac{\tilde{h}^{ab}(1-|\nabla h|^2)-(|\nabla h|^2+\tilde{h}^{cd}\partial_ch\,\partial_dh)g^{ab}}{(1+\la \N_F,\td \N\ra)^2}\partial_a\partial_bh\cdot\N_F.
\end{align}

\section{\bf Constructions of doubling of asymptotically conical shrinkers} \label{double}\label{sec:3}
We are going to construct a closed hypersurface which is a double of a large compact set of the concerned asymptotically conical shrinker $\Sigma$.
The main idea is similar to \cite{Sto}, but here the process will be more technical since we are working on extrinsic hypersurfaces.

As in \eqref{cone-C}, denoting the asymptotic cone of $\Sigma$ by $\mathcal C$, from \cite{W14,CS}, we know that outside a large compact set $\Sigma\cap \ovl B_R$, we can parameterize $\Sigma\setminus B_R$ by $$F_0(r,\theta)=r\theta+f(r,\theta)\N$$ 
over $\mathcal C\setminus B_R$, where $\theta$ is the parametrization for the link of $\mathcal C$ and $\N$ is the normal vector of $\mathcal C$ at $(r,\theta)$.
We will modify $\Sigma$ in this region, and glue it with the compact part.

We choose a smooth cutoff function 
$\eta\colon\RR\to \RR$ 
such that $\eta=\left\{\begin{array}{cc}
     1,\,\,x<0  \\
     0,\,\,x>1 
\end{array}\right..$ 
Take 
$F_1(r,\theta)=r\theta+\eta(\frac{r-R}{R})f(r,\theta)\N$. 
Note that to glue two copies, we need to modify the end such that it becomes cylindrical. 

The first step is to move the link to a hemisphere. 
As the link of $\mathcal C$ is a codimension one submanifold of $S^{n-1}$, we can pick a vector $\mathbf v\in S^{n-1}$ such that $-\mathbf v$ is not contained in the link of $\mathcal C$. 
Initially, we have $F_1(r,\theta)=r\theta$ for $r>2R.$ 
We take 
    $$F_2(r,\theta)
    = r \cdot \frac{\pr{\frac{2}{3}\eta\pr{\frac{r-2R}{R}}+\frac{1}{3}}\theta + \pr{\frac{2}{3}-\frac{2}{3}\eta\pr{\frac{r-2R}{R}}}\mathbf v}
    {\abs{\pr{\frac{2}{3}\eta\pr{\frac{r-2R}{R}}+\frac{1}{3}}\theta + \pr{\frac{2}{3}-\frac{2}{3}\eta\pr{\frac{r-2R}{R}}}\mathbf v}}$$ 
so that for $r\geq 3R,$ the image is contained in a hemisphere. 
Consequently, we can define 
\begin{align*}
    F_3(r,\theta)
    = \begin{cases}
        r\theta+\eta(\frac{r-R}{R})f(r,\theta)\N
        &\text{if }R\leq r\le 2R\\
        r\cdot \frac{\pr{\frac{2}{3}\eta\pr{\frac{r-2R}{R}}+\frac{1}{3}}\theta + \pr{\frac{2}{3}-\frac{2}{3}\eta\pr{\frac{r-2R}{R}}}\mathbf v} {\abs{\pr{\frac{2}{3}\eta\pr{\frac{r-2R}{R}}+\frac{1}{3}}\theta + \pr{\frac{2}{3}-\frac{2}{3}\eta\pr{\frac{r-2R}{R}}}\mathbf v}}
        &\text{if }r\ge 2R
    \end{cases}.
\end{align*}

The next step is to modify the end such that the radial derivative is in the same direction. 
To this end, we take
\begin{align*}
    F_4(r,\theta)
    = \begin{cases}
        F_3(r,\theta)
        &\text{if }R \le r\le 3R\\
        3R\cdot \frac{\theta+2\mathbf v}{|\theta+2\mathbf v|}
        +\int_{3R}^r \pr{\eta(\frac{r-3R}{R})\frac{\theta+2\mathbf v}{|\theta+2\mathbf v|}+(1-\eta(\frac{r-3R}{R}))\mathbf v} ds
        &\text{if }3R\leq r\leq 4R
    \end{cases}.
\end{align*}

The last step is to make the end cylindrical and end at the same plane so that we can do a reflection to double it. 
Denote $B:=\sup_{\theta}\la F_4(4R,\theta),\,\mathbf v\ra,$ 
and consider 
\begin{align*}
    F(r,\theta)
    = \begin{cases}
        r\theta+\eta(\frac{r-R}{R})f(r,\theta)\N
        &\text{if }R \leq r\le 2R\\
        r\cdot \frac{\pr{\frac{2}{3}\eta\pr{\frac{r-2R}{R}}+\frac{1}{3}}\theta + \pr{\frac{2}{3}-\frac{2}{3}\eta\pr{\frac{r-2R}{R}}}\mathbf v} {\abs{\pr{\frac{2}{3}\eta\pr{\frac{r-2R}{R}}+\frac{1}{3}}\theta + \pr{\frac{2}{3}-\frac{2}{3}\eta\pr{\frac{r-2R}{R}}}\mathbf v}}
        &\text{if }2R\leq r\leq 3R\\
       3R\cdot \frac{\theta+2\mathbf v}{|\theta+2\mathbf v|}
        + \int_{3R}^r \pr{\eta(\frac{r-3R}{R})\frac{\theta+2\mathbf v}{|\theta+2\mathbf v|}+(1-\eta(\frac{r-3R}{R}))\mathbf v} ds
        &\text{if }3R\leq r\leq 4R\\
        F_4(4R,\theta) + \pr{\eta(\frac{r-4R}{R})-1}\la F_4(4R,\theta),\mathbf v\ra\mathbf v
        +\pr{1-\eta(\frac{r-4R}{R})} (B+R)\mathbf v
        &\text{if }4R\leq r\leq 5R
    \end{cases}.
\end{align*}

We can now glue $F$ with the original shrinker along $F(R,\cdot)$ with $\Sigma\cap \ovl B_R$. 
As $\eta$ is a cutoff function, we have that the glued hypersurface is also smooth. 
From the constructions, the $F_3$ part lies in a hemisphere and controls the distance to origin. 
Thus, we know that $F$ is globally an embedding, and after reflecting it with respect to the hyperplane
$\{x\in \RR^{n+1}\,|\, \la x, \mathbf v\ra = B+R\},$ 
we get a closed embedded hypersurface.

\section{\bf Interior estimates}\label{sec:4}
We are going to establish some interior estimates, as in the Ricci flow case developed in \cite{Sto}.
To get the required H\"older estimates in the case of mean curvature flows, we will modify the estimates in \cite[Proposition 2.5]{Bam}, \cite[Lemma 4.4]{App}, and \cite[Proposition C.4.]{Sto}. 
For completeness, we sketch their proofs here.

\begin{Pro}[H\"older Estimate. For the discussions in Ricci flow settings, see \cite{Bam, App, Sto}.]
\label{GHolder}
Take $2r\geq s>0$ and consider $\Omega=B_r\times [0,s^2]$ and $\Omega'=B_{2r}\times [0,s^2]$. Assume $u\in C^2(\Omega')$ satisfies 
\begin{align*}
    (\partial_t-L)u=Q[u]+F=&r^{-2}f_1(r^{-1}x,r^{-2}t,u,r\nabla u)\cdot u+r^{-1}f_2(r^{-1}x,r^{-2}t,u,r\nabla u)\cdot \nabla u\\ 
    &+ f_3(r^{-1}x,r^{-2}t,u,r\nabla u)\cdot \nabla u\otimes\nabla u+f_4(r^{-1}x,r^{-2}t,u,r\nabla u)\cdot u\otimes\nabla^2 u+F(x,t)
\end{align*}
with $u(\cdot,0)=u_0$, where $f_1,f_2,f_3,$ and $f_4$ are smooth. 
The operator $L$ is in the form 
    \begin{align*}
        Lu =a^{ij}(x,t)\partial^2_{ij}u+b^i(x,t)\partial_iu+c(x,t)u
    \end{align*}
and the coefficients satisfy the following bounds, for $m\geq 1,$
\begin{align*}
    \frac{1}{\Lambda}\leq a^{ij}\leq \Lambda,\,
    \|a^{ij}\|_{C^{2m-2,\alpha}(\Omega')}\leq \Lambda,\,
    \|b^{i}\|_{C^{2m-2,\alpha}(\Omega')}\leq r^{-1}\Lambda,\,\text{ and }
    \|c\|_{C^{2m-2,\alpha}(\Omega')}\leq r^{-2}\Lambda.
\end{align*}
Then there exist constants $\ep>0$ and $C<\infty$ depending on $n,m,\alpha,\Lambda,$ and $f_i$ such that if 
\begin{align*}
    \|u\|_{c^0(\Omega')}+\|u_0\|_{C^{2m,\alpha}(B_{2r})}+r^2\|F\|_{C^{2m-2,\alpha}(\Omega')}
    \leq \ep,
\end{align*}
then 
\begin{align*}
    \|u\|_{C^{2m,\alpha}(\Omega)}
    \leq C\pr{\|u\|_{C^0(\Omega')}+\|u_0\|_{C^{2m,\alpha}(B_{2r})}+r^2\|F\|_{C^{2m-2,\alpha}(\Omega')}}.
\end{align*}
\end{Pro}

\begin{proof}
    As in \cite{Sto}, we define the weighted norm as 
    \begin{align*}
        \|u\|^{(\theta)}_{C^{2m,\alpha}(\Omega)}:=\sum_{i+2j\leq 2m}(r_\Omega\theta)^{i+2j}(\|\na^i\partial^j_tu\|_{C^0(\Omega)}+(r_\Omega\theta)^\alpha[\na^i\partial^j_tu]_{\alpha,\frac{\alpha}{2}}).
    \end{align*}
Then, from the standard parabolic estimates (cf. \cite{Kry}), we have  
\begin{align*}
    \|u\|_{C^{2m,\alpha}(B_r\times[0,s^2])}
    \leq C\pr{(r\theta)^2\|Q[u]+F\|_{C^{2m-2,\alpha}(B_{(1+\theta)r}\times [0,s^2])}^{(\theta)}+\|u\|^{(\theta)}_{C^0(B_{(1+\theta)r}\times [0,s^2])}+\|u_0\|^{(\theta)}_{C^{2m,\alpha}(B_{(1+\theta)r})}}.
\end{align*}
After rescaling, we may assume $r=1$. 
Define \begin{align*}
    H& := \|u\|_{C^0(B_2\times[0,s^2])} 
    + \|u_0\|_{C^{2m,\alpha}(B_2)}+\|F\|_{C^{2m-2,\alpha}(B_2\times [0,s^2])},\\  
    r_k&= 2-2^{-k},\,
    \theta_k = \frac{r_{k+1}-r_k}{r_k},\,
    \Omega=B_{r_k}\times[0,s^2],\text{ and}\\
    a_k&:=\|u\|^{(\theta_k)}_{C^{2m,\alpha}(\Omega_k)}
    \leq C\pr{\theta_k^2\|Q[u]\|^{(\theta_{k+1})}_{C^{2m-2,\alpha}(\Omega_{k+1})}+H}.
\end{align*}
For $i=1,\cdots,4,$ we have
\begin{align*}
    \|f_i(x,t,u,r\na u)\|_{C^{2m-2,\alpha}(\Omega_{k+1})}^{(\theta_{k+1})}
    \leq C\pr{1+(\|u\|_{C^{2m,\alpha}(\Omega_{k+1})}^{(\theta_{k+1})})^{2m-1}}
    \leq C\pr{1+a_{k+1}^{2m-1}},
\end{align*}
so by plugging in the estimates, we get 
\begin{align*}
    a_k
    \leq C\pr{\theta_{k+1}a_{k+1}+a_{k+1}^2+a_{k+1}^{2m}+a_{k+1}^{2m+1}+H}.
\end{align*}
Take $k_0$ such that for $k\geq k_0$, $\theta_{k+1}\leq \frac{1}{16C_0^2}$. 
Then for $b_k:=\frac{a_k}{H}$, $K\geq k_0,$ and $H\leq \ep,$ we have 
\begin{align*}
    b_k\leq \frac{1}{16}b_{k+1}+\frac{1}{16C_0}b_{k+1}^2+\frac{1}{16^{2m-1}C_0^{4m-3}}b_{k+1}^{2m}+\frac{1}{16^{2m}C_0^{4m-1}}b_{k+1}^{2m+1}+C_0.
\end{align*}
As $\lim_{k\to\infty}b_k\leq 1<2C_0$, by induction, we know that $b_{k_0}<2C_0$. 
Hence $a_0\leq C'H$, which completes the proof.   
\end{proof}

From Proposition~\ref{GHolder}, we have the following estimate.

\begin{Pro}[For the discussion in Ricci flow settings, see  {\cite[Lemma 6.2]{Sto}}] \label{hHolder}
Given $h$ satisfying \eqref{h-equ} on $\Omega'\times (\tau_0,\tau_2)$, let $\Omega\Subset \Omega'\Subset \Sigma$, $m\in\mathbb{N}$, $\tau_0\leq \tau_1< \tau_2,$ and $\mathcal{Q}_1=0$ on $\Omega'\times[\tau_0,\tau_2]$. 
Then there exist constants $\ep = \ep(n,m)>0$ and $C = C\pr{n,\Sigma,F,\dist(\Omega,M\setminus\Omega'),\tau_1-\tau_0} < \infty$ 
with the following estimates.\\
(1) If $\tau_1>\tau_0$, then 
\begin{align*}
    \|h\|_{L^\infty(\Omega'\times[\tau_0,\tau_2])}<\ep\Rightarrow \|h\|_{C^m(\Omega\times[\tau_1,\tau_2])}\leq C\|h\|_{L^\infty(\Omega'\times[\tau_0,\tau_2])}.
\end{align*}
(2) If $\tau_1=\tau_0$, then
\begin{align*}
    \|h\|_{L^\infty (\Omega'\times[\tau_0,\tau_2])} + \|h\|_{C^{m+1}(\Omega'\times\{\tau_0\})}<\ep\Rightarrow \|h\|_{C^m(\Omega\times[\tau_0,\tau_2])}
    \leq C\pr{\|h\|_{L^\infty(\Omega'\times[\tau_0,\tau_2])}+\|h\|_{C^{m+1}(\Omega'\times\{\tau_0\})}}.
\end{align*}
\end{Pro}

On the other hand, we also have interior integral estimates.
Recall that we use $L^2_W$ and $H^k_W$ to be the weighted $L^2$ and $H^k$ spaces with the weight given by $e^{-\frac{|F|^2}4}.$

\begin{Lem}[(For the discussion in Ricci flow settings, see {\cite[Lemma 6.3]{Sto}}]
    \label{hIntegral}
Given $h$ satisfying \eqref{h-equ} on $\Omega'\times (\tau_0,\tau_2)$, let $\Omega\Subset \Omega'\Subset \Sigma$, $m\in\mathbb{N}$, $\tau_0\leq \tau_1< \tau_2,$ and $\mathcal{Q}_1=0$ on $\Omega'\times[\tau_0,\tau_2]$. 
Then there exist constants $\ep = \ep(n,m)>0$ and $C = C(n,\Sigma,F,\dist(\Omega,M\setminus\Omega'),\tau_1-\tau_0)<\infty$ 
with the following estimates.\\
(1) If $\tau_1>\tau_0$, then 
\begin{align*}
    \|h\|_{C^{m+2}(\Omega'\times[\tau_0,\tau_2])}<\ep\Rightarrow \sup_{\tau\in[\tau_1,\tau_2]}\|\na^mh(\tau)\|_{L^2_W(\Omega)} + \|h\|_{H_W^{m+1}(\Omega\times[\tau_1,\tau_2])}
    \leq C\|h\|_{H_W^m(\Omega'\times[\tau_0,\tau_2])}.
\end{align*}
(2) If $\tau_1=\tau_0$, then
\begin{align*}
 &\|h\|_{C^{m+2}(\Omega'\times[\tau_0,\tau_2])}<\ep\\ \notag\Rightarrow &\sup_{\tau\in[\tau_0,\tau_2]}\|\na^mh(\tau)\|_{L^2_W(\Omega)}+\|h\|_{H_W^{m+1}(\Omega\times[\tau_0,\tau_2])}
 \leq C\pr{\|h(\tau_0)\|_{H_W^m(\Omega')}+\|h\|_{H_W^m(\Omega'\times[\tau_0,\tau_2])}}.
\end{align*}
\end{Lem}

\begin{proof}
    The proof is essentially the same as \cite[Lemma 6.3]{Sto}. 
    The only difference is that we need to assume that one more order derivative is bounded as the coefficient of \eqref{Q_2second} is more nonlinear.

From \eqref{h-equ}, \eqref{Q_2second}, and the assumption $\mathcal{Q}_1=0$ on $\Omega'\times[\tau_0,\tau_2]$, we have 
\begin{align*}
    \partial_th
    = \dL h +\sum_{l=1}^1B_l^0*\na^lh-{\rm div}_{\frac{|F|^2}{4}}\pr{\hat{h}^0*\na h},
\end{align*}
where $\dL = \Delta_{\frac{|F|^2}4}$ is the drift Laplacian and there exists $C_0,$ and for any $\delta>0$, there exists $\ep>0$ such that 
\begin{align*}
    \|h(\tau)\|_{C^2(\Omega')}<\ep\Rightarrow \|B_1^0\|_{L^\infty(\Omega')}+\|B_0^0\|_{L^\infty(\Omega')}<C_0,\,\|\hat{h}^0\|_{L^\infty(\Omega')}\leq \delta.
\end{align*}
By induction, we have 
\begin{align*}
    \partial_t\na^mh
    = \dL \na^mh +\sum_{l=1}^{m+1}B_l^m*\na^lh-{\rm div}_{\frac{|F|^2}{4}}\pr{\hat{h}^m*\na^{m+1} h},
\end{align*}
where there exists $C_0$ and for any $\delta>0$, there exists $\ep>0$ such that 
\begin{align*}
    \|h(\tau)\|_{C^{m+2}(\Omega')}<\ep\Rightarrow \sum_{l=0}^{m+1}\|B_l^m\|_{L^\infty(\Omega')}<C_m,\,\|\hat{h}^m\|_{L^\infty(\Omega')}\leq \delta.
\end{align*}
Then the rest of the proof follows from integration by parts and Young's inequality, cf. \cite[Lemma 6.3]{Sto}.
\end{proof}

Finally, we arrive at the local interior estimate we will use later.

\begin{Pro}[For the discussion in Ricci flow settings, see {\cite[Lemma 6.4]{Sto}}]\label{int1}
    Given $\Gamma>0$ and $h$ satisfying \eqref{h-equ} on $\Sigma\times (0,\tau_1)$, assume that $\mathcal{Q}_1=0$ on $\set{\frac{|F|^2}{4}<2\Gamma e^{\tau_0}} \times[0,\tau_1]$ and
    \begin{align*}
        \|h\|_{C^{0,\alpha}_{\rm hom,-1}(\Sigma\times [0,\tau_1])}&<\ep,\\
        \|h(\tau)\|_{L^2_W(\Sigma)}&\leq \mu e^{-\lambda_*(\tau+\tau_0)}\,\,\text{ for all}\,\,\tau\in[0,\tau_1], \text{ and}\\
        \sum_{l=0}^{4n+3} \abs{\na^lh}_g(x,0)&\leq C_0\bar{p}e^{-\lambda_*\tau_0}\,\,\text{ for all}\,\,x\in \set{\frac{|F|^2}{4}<2\Gamma e^{\tau_0}}\text{ and }\tau=0
    \end{align*}
    for some $\lambda_*>0.$
    Then for $\Gamma>\Gamma_0(n,\Sigma,F)$, $\ep<\ep_0(n,\Sigma,F,\Gamma),$ and $\tau_0>C(n,\Sigma,F,\lambda_*,\Gamma,C_0\bar{p})$, there exists $C_1=C_1(n,\Sigma,F,\lambda_*,\Gamma)<\infty$ such that 
    \begin{align*}
        \sum_{l=0}^2|\na^g h|_{g}(x,\tau)\leq C_1(\mu+C_0\bar{p})e^{-\lambda_*(\tau+\tau_0)}\,\,\text{ for all}\,\,(x,\tau)\in \set{\frac{|F|^2}{4}<\Gamma e^{\tau_0}}\times [0,\tau_1].
    \end{align*}
\end{Pro}

Note that we can also use Proposition~\ref{GHolder} to get large scale interior estimates by rescaling. 
More precisely, consider a family of diffeomorphisms $\varphi_s\colon M\to M$ given by
\begin{align*}
    \partial_s\varphi_s
    = \na^g\pr{\frac{|F|^2}{4}}\circ \varphi_s
\end{align*}
and define 
$\mathcal{H}(t)=\varphi_t^*h(t).$
The evolution equation of $\mathcal{H}$ is then
\begin{align}\label{H-equ}
    \frac{\partial \mathcal{H}}{\partial t}(s)
    =&\pr{\Delta^g+|A|^2+\frac{1}{2}}\mathcal{H}(s) + \varphi_s^*\mathcal{Q}_1 + \varphi_s^*\mathcal{Q}_2.
\end{align}
This helps us to get rid of the unbounded coefficient term $-\frac{1}{2}\la F,\nabla^g f\ra$ in $L$. 
Also, note that the relation
$0\leq \abs{\na^g(\frac{|F|^2}{4})}^2
\leq \frac{|F|^2}{4}$ 
implies that for $s_1<s_2,$
\begin{align*}
\frac{|F|^2}{4}(\varphi_{s_1}(x))\leq \frac{|F|^2}{4}(\varphi_{s_2}(x))
\end{align*}
and
\begin{align*}
e^{-s_1}\frac{|F|^2}{4}(\varphi_{s_1}(x))\geq e^{-s_2}\frac{|F|^2}{4}(\varphi_{s_2}(x)).
\end{align*}
Moreover, from \cite[Lemma 3.8]{CS}, if we denote the link of the asymptotic cone of $\Sigma$ by $\mathcal{C}$, by parameterizing the shrinker outside a compact set by $F:(R,\infty)\times \Gamma\to \Sigma$  and denoting the family of diffeomorphisms generated by $X_t=\frac{1}{2(-t)}x^T$ by $\Phi_t$, 
we can define $\tilde{\Phi}_t=F^{-1}\circ \Phi_t\circ F$ and $\phi_t:(R,\infty)\times \Gamma\to (R,\infty)\times \Gamma,\,(r,\omega)\to ((-t)^{-\frac{1}{2}}r,\omega)$ with the following estimates.

\begin{Lem}[{\cite[Lemma 3.8]{CS}}]\label{comp}
    \begin{align*}
d_{g_{\mathcal{C}}}(\tilde{\Phi_t}(t,\theta),\phi_t(r,\theta))\lesssim\frac{1}{\sqrt{-t}\cdot r},\\
|D^{(j)}(\tilde{\Phi}-\phi_t)|(r,\theta)\lesssim \frac{1}{\sqrt{-t}\cdot r^{1+j-\eta}},
\end{align*}
for $j\geq 1$ and $\eta>0$.
\end{Lem}

As in \cite{CS}, we recall a non-standard interior Schauder estimates from \cite{Bra69}.

\begin{Lem}[Non-standard interior Schauder estimates \cite{Bra69,CS}]\label{nonsch}
Suppose $B_2\in \bb R^n$ and there are $C^{0,\alpha}$ functions $a_{ij},b_i,c,u,f\colon B_2\times [-2,0]\to\bb R$ such that
	$$\frac{\bd}{\bd t}u
	= a_{ij} D_{ij}^2 u
	+ b_i D_iu
	+ cu
	+ f.
	$$
	If the system is uniformly parabolic in the sense that
	$a_{ij}\xi_i\xi_j\ge \lambda|\xi|^2$
	for some $\lambda>0$ and 
	$$\sup_{t\in[-2,0]}
	\pr{
		\|a_{ij}(\cdot,t)\|_{C^{0,\alpha}(B_1)}
		+ \|b_{i}(\cdot,t)\|_{C^{0,\alpha}(B_1)}
		+ \|c(\cdot,t)\|_{C^{0,\alpha}(B_1)}
	}
	\le \Lambda$$
	for some $\Lambda>0,$ then for $T\in (-1,0],$
	$$
	\sup_{t\in[-1,T]} \|u(\cdot,t)\|_{C^{2,\alpha}(B_1)}
	\le C\sup_{t\in[-2,T]} \pr{
		\|u(\cdot,t)\|_{C^{0,\alpha}(B_2)}
		+ \|f(\cdot,t)\|_{C^{0,\alpha}(B_2)}
	}
	$$
	for some $C=C(n,\alpha,\lambda,\Lambda)$.
\end{Lem}

Combining Lemma~\ref{comp}, Lemma~\ref{nonsch} and Proposition~\ref{GHolder} implies the large scale interior estimates.

\begin{Lem}[For the discussion in Ricci flow settings, see {\cite[Lemma 6.8]{Sto}}]\label{int2}
      Given $h$ satisfying \eqref{h-equ} and $\mathcal{Q}_1=0$ on $\set{(x,\tau)\in\Sigma\times (0,\tau_1):\,\Gamma e^{\tau_0}<\frac{|F|^2}{4}<\gamma e^{\tau+\tau_0}},$ assume that 
    \begin{align*}
        |h|_g&\leq C_0e^{-k(\tau+\tau_0)} \pr{\frac{|F|^2}{4}(x)}^{k+\frac{1}{2}}\,\,\text{on}
        \,\,\set{(x,\tau)\in\Sigma\times (0,\tau_1):\,\Gamma e^{\tau_0}<\frac{|F|^2}{4}<\gamma e^{\tau+\tau_0}}, \text{ and}\\
        |\na^lh|_g(x,\tau_0)&\leq C_0e^{-k\tau_0} \pr{\frac{|F|^2}{4}(x)}^{k+\frac{1}{2}-\frac{l}{2}}\,\,\text{for all}\,\,
        l\in\{0,1,2,3\},\,\,x\in \set{\Gamma e^{\tau_0}<\frac{|F|^2}{4}<\gamma e^{\tau+\tau_0}}.
    \end{align*}
    Then for $\Gamma>\Gamma_0(n,\Sigma,F)$ and $C_0\gamma^k<C(n,\Sigma,F)$, there exists $C_1=C_1(n,\Sigma,F,k)<\infty$ such that 
    \begin{align*}
        |\na^l h|_{g}(x,\tau)\leq C_1C_0e^{-k(\tau+\tau_0)}\pr{\frac{|F|^2}{4}(x)}^{k+\frac{1}{2}-\frac{l}{2}}\,\,\text{on}\,\,\set{(x,\tau)\in\Sigma\times [0,\tau_1]:\,4\Gamma e^{\tau_0}\leq\frac{|F|^2}{4}(x)\leq\frac{1}{4}\gamma e^{\tau+\tau_0}}
    \end{align*}
    for all $l\in\{0,1,2\}.$
\end{Lem}

Similarly, in the outer region, we have the following estimate.
\begin{Lem}[For the discussion in Ricci flow settings, see {\cite[Lemma 6.9]{Sto}}]\label{int3}
      Given $h$ satisfying \eqref{h-equ} on $\Sigma\times (0,\tau_1),$ assume that 
    \begin{align*}
        |h|_g(x,\tau)&\leq \delta|F|\,\,
        \text{on}\,\,\set{(x,t)\in\Sigma\times [0,\tau_1]:\,\gamma e^{\tau+\tau_0}<\frac{|F|^2}{4}(x)},\\
        |\na^lh|_g(x,0)&\leq\delta|F|^{1-l},\,\,l\in\{0,1,2,3\}\,\,
        \text{on}\,\,x\in \set{\gamma e^{\tau_0}\leq\frac{|F|^2}{4}},\text{ and}\\
        |\mathcal{Q}_1|_g(x,\tau)) + |\na^g\mathcal{Q}_1|_g(x,\tau)&\leq \delta\,\,
        \text{on}\,\,\set{(x,\tau)\in\Sigma\times [0,\tau_1]:\,\gamma e^{\tau+\tau_0}<\frac{|F|^2}{4}(x)}.
    \end{align*}
    Then for $\delta<\delta_0(n,\Sigma,F)$ and $\tau_0>\tau(n,\Sigma,F,\gamma)$, there exists $C=C(n,\Sigma,F)<\infty$ such that 
    \begin{align*}
        |\na^l h|_{g}(x,\tau)\leq C\delta|F|^{1-l},l\in\{0,1,2\}\,\,\text{on}\,\,\set{(x,\tau)\in\Sigma\times [0,\tau_1]:\,4\gamma e^{\tau+\tau_0}<\frac{|F|^2}{4}(x)}.
    \end{align*}
\end{Lem}

\section{\bf Barrier functions}\label{sec:5}
In this section, we will use a barrier function argument to get the corresponding $C_0$ estimate. 
Since the original evolution equation \eqref{h-equ} is highly nonlinear, to construct barriers, we have to find the correct operator first. 
We will consider similar barrier functions as in the Ricci flow case \cite{Sto} while we need to handle the nonlinearity of mean curvature flows. 
Recall that in normal coordinates, the second order term in $\mathcal{Q}_2$ is 
\begin{align}\label{1storder}
C^{ab}(F,\na^gF,h,\na ^gh)\partial_a\partial_bh\cdot\N_F
:=\frac{\tilde{h}^{ab}\pr{1-|\nabla h|^2} - \pr{|\nabla h|^2+\tilde{h}^{cd}\partial_ch\,\partial_dh} g^{ab}}{1+\la \N_F,\td \N\ra}\partial_a\partial_bh\cdot\N_F.
\end{align}
We look at the lower order terms of $\mathcal{Q}_2$ in \eqref{h-equ}. 
Using \eqref{nota}, if we denote $C=C(F,\na F,\na^2F,h,\na h)$ to be any tensor depending on $F,\na F,\na^2F,h,$ and $\na h$,
we get 
\begin{align*}
   &C_3h+C_4*\na h\\
   =&C*(\na h \N_F+hA)*(\na h \N_F+hA)+C*\na h*\na h* A+C*hA*A\\ 
   &+ hC*(\na h \N_F+hA)*(\na h \N_F+hA)\\ 
   &+ (h+C)(C*\na h+C*\na h \N_F+C*hA)*(\na h+C*\na h \N_F+C*hA)\\ 
   &+ (ch*A+C*\na h*\na h)(C*\na h*A+hC*\na A+hC*A*A)\\ 
   & +\na h*(h*\na A+A*\na h+h\na h*A*A+h^2A*\na A)*(C+h*A+C*\na h*\na h+C*hA),
\end{align*}
and we can simplify the expression to get
\begin{align*}
     C_3h+C_4*\na h=C*\na h*\na h+C*hA*A+C*\na h*hA+Ch*\na A.
\end{align*}
Thus, we know 
\begin{align*}
    | C_3h+C_4*\na h|\leq C|\na h|^2+C|h||A|^2+C|h||\na A|,
\end{align*}
so we have 
\begin{align*}
    \partial_\tau|h|^2
    &=2\la h\,\N_F,\partial_th\,\N_F\ra = 2h(Lh+\mathcal{Q}_1+\mathcal{Q}_2)\\ 
    &\leq \pr{\Delta-\na_{\na^g \frac{|F|^2}{4}}} |h|^2 - 2|\na h|^2
    + \pr{2|A|^2+1}h^2
    + C^{ab}\na_a\na_bh^2-2C^{ab}\na_ah\na_bh\\  
    &\quad +2h\mathcal{Q}_1
    +h \pr{C|\na h|^2+C|h||A|^2+C|h||\na A|}.
\end{align*}
Using \eqref{1storder}, given 
$\|h\|_{C^{1,\alpha}_{\rm hom,-1}}\leq \ep<\ep_n,$
we have 
\begin{align*}
    |C^{ab}|\leq C_n\pr{|hA|_{C^0}+|\na h|^2_{C^0}},
\end{align*}
and
\begin{align*}
     \partial_\tau|h|^2
     &\leq \pr{g^{ab}+C^{ab}} \na_a\na_bh^2
     -\na_{\na^g \frac{|F|^2}{4}}|h|^2 +\pr{C_n|A|^2+C_n|\na A|+1}h^2
     +2|h||\mathcal{Q}_1|.
\end{align*}
These imply that we can look at the operator 
\begin{align}\label{oper}
    Qf
    :=\pr{g^{ab}+C^{ab}}\na_a\na_bf
    -\na_{\na^g \frac{|F|^2}{4}}f
    + \pr{C_n|A|^2+C_n|\na A|+1}f
    +2\sqrt{|f|}|\mathcal{Q}_1|.
\end{align}
We now take $f=\frac{|F|^2}{4}.$ 
From the shrinker equation, we have 
$\Delta \frac{|F|^2}{4}-\na_{\na^g \frac{|F|^2}{4}}\frac{|F|^2}{4}=\frac{n}{2}-f$. 
Hence, we obtain
\begin{align*}
    \pr{g^{ab}+C^{ab}} \na_a\na_bf^k
    -\na_{\na^g \frac{|F|^2}{4}}f^k
    = kf^{k-1}\pr{\frac{n}{2}-f} 
    + k(k-1)f^{k-1}\frac{|\na f|^2}{f}
    +C^{ab}\na_a\na_bf^k.
\end{align*}
As 
\begin{align}\label{Cest}
    \abs{C^{ab}\na_a\na_bf^k}
    & = \abs{kf^{k-1}C^{ab}\na_a\na_bf+k(k-1)C^{ab}f^{k-2}\na_af\na_bf} \\
    & \leq \abs{kf^{k-1}} \cdot
    \abs{C^{ab}}
    \pr{|A_{ab}|f^{\frac{1}{2}}+(k-1)|F_a||F_b|},\nonumber
\end{align}
similarly as in \cite{Sto}, we take 
\begin{align*}
    u=\begin{cases}
        e^{(-k+1)(\tau+\tau_0)} \pr{Df^k-Bf^{k-1}}
        &\text{on }
        \set{(x,\tau)\in \Sigma\times[0,\infty):\Gamma<f(x)<\gamma_1 e^{\tau+\tau_0}}\\
        e^{\tau+\tau_0}\pr{a-Ef^{-1}}
        &\text{on }
        \set{(x,\tau)\in \Sigma\times[0,\infty):\gamma_2 e^{\tau+\tau_0}<f(x)<\Gamma e^{\tau+\tau_0}}
    \end{cases}
\end{align*}
for some positive numbers $a,B, D,$ and $E.$

In general, using \eqref{Cest} and $0\leq\frac{|\na f|^2}{f}\leq 1,$ we calculate 
\begin{align*}
   & e^{(m-1)(\tau+\tau_0)}(\partial_\tau-Q)u\\
   =& e^{(m-1)(\tau+\tau_0)}(\partial_\tau-Q) \pr{e^{(-m+1)(\tau+\tau_0)}\pr{Df^m-Bf^{m-1}}}\\ 
   =&-\frac{n}{2} \pr{mDf^{m-1}-(m-1)Bf^{m-2}} + Bf^{m-1} -(m(m-1) Df^{m-1} - (m-1)(m-2)Bf^{m-2}) \frac{|\na f|^2}{f}\\  
   &- C^{ab} \na_a\na_b \pr{Df^m-Bf^{m-1}} -\pr{C_n|A|^2+C_n|\na A|} \pr{Df^m-Bf^{m-1}} -e^{(m-1)(\tau+\tau_0)} 2\sqrt{|u|}|\mathcal{Q}_1|\\ 
   \ge& f^{m-1}\left(B + \frac{n(m-1)B}{2f} -\frac{nmD}{2} - m(m-1)D -Dm \abs{C^{ab}} \abs{A_{ab}} f^{\frac{1}{2}} - Dm(m-1)\abs{C^{ab}}\right.\\ 
   &\left.-B(m-1) \abs{C^{ab}} \abs{A_{ab}} f^{-\frac{1}{2}} - \frac{B(m-1)(m-2)\abs{C^{ab}}}{f} - \pr{C_n|A|^2+C_n|\na A|} (Df+B)) \right)\\ 
   &-e^{(m-1)(\tau+\tau_0)}2\sqrt{|e^{(-m+1)(\tau+\tau_0)}(Df^m-Bf^{m-1})|}|\mathcal{Q}_1|.
\end{align*}
Therefore, as $\mathcal{Q}_1$ is 0 on 
$\set{(x,\tau)\in \Sigma\times[0,\infty):\Gamma e^{\tau_0}<f(x)<\gamma_1 e^{\tau+\tau_0}}$, there exist constants $\ep_n,\Gamma,$ and $\gamma_1$ such that given 
\begin{align*}
    &\|h\|_{C^{1,\alpha}_{\rm hom,-1}}\leq \ep<\ep_n,\quad 
    B-D\pr{m\pr{m+\frac{n}{2}} + C_n\sup_\Sigma\pr{f|A|^2+f|\na A|}}
    \geq w>0, \\
    &\Gamma \geq \Gamma(n,B,m,w),
    \text{ and }\gamma_1\leq \gamma_1(n,\Sigma,F,B,D,m,w),
\end{align*}
$u$ satisfies $(\partial_t-Q)u\geq 0$ on 
$\set{(x,\tau)\in \Sigma\times[0,\infty):\Gamma e^{\tau_0}<f(x)<\gamma_1 e^{\tau+\tau_0}}$.

Moreover, if 
\begin{align*}
    \mathcal{Q}_1\leq \frac{C_0}{\Gamma}e^{-\frac{\tau+\tau_0}{2}},
\end{align*}
then on 
$\set{(x,\tau)\in \Sigma\times[0,\infty):\gamma_2 e^{(\tau+\tau_0)}<f(x)<\Gamma e^{(\tau+\tau_0)}},$
\begin{align*}
    e^{-(\tau+\tau_0)}2\sqrt{|e^{(\tau+\tau_0)}(a-Bf^{-1})|}
    \abs{\mathcal{Q}_1}
    \leq \frac{2\sqrt{a}}{e^{(\tau+\tau_0)}\Gamma}
    \leq 2\frac{\sqrt{a}}{f}.
\end{align*}
Plugging in the expression, there exists $\tau_0$ such that given 
\begin{align*}
    &|h|_{C^{\rm hom}_{1;1}}\leq \ep<\ep_n,\quad 
    B - a\pr{C_n\sup_\Sigma \pr{f|A|^2+f|\na A|}}-C_0\sqrt{a}\geq w>0,\text{ and}\\
    &\tau_0\geq \tau_0(n,\Sigma,F,C_0,B,\gamma_2,\ep_n),
\end{align*}
$u$ satisfies $(\partial_\tau-Q)u\geq 0$ on 
$\set{(x,\tau)\in \Sigma\times[0,\infty):\gamma_2 e^{(\tau+\tau_0)}<f(x)<\Gamma e^{(\tau+\tau_0)}}$.

Therefore, we have the following estimates.

\begin{Pro}[For the discussion in Ricci flow settings, see {\cite[Lemmas 6.5,6.6,6,7]{Sto}}]\label{barrier}
    Given $|h|_{C_1}\leq \ep<\ep_n,$ 
    we have $(\partial_\tau-Q)h^2\leq 0,$ 
    where $Q$ is given by \eqref{oper}.
Moreover, if $\mathcal Q_1$ satisfies
\begin{align*}
    \abs{\mathcal Q_1}
    \begin{cases}
        =0
        & \text{on }\set{(x,\tau)\in \Sigma\times[0,\infty):\Gamma e^{\tau_0}<f(x)<\gamma_1 e^{(\tau+\tau_0)}}\\
        \leq \frac{C_0}{\Gamma}e^{-\frac{(\tau+\tau_0)}{2}}
        & \text{on }\set{(x,\tau)\in \Sigma\times[0,\infty):\gamma_2 e^{(\tau+\tau_0)}<f(x)<\Gamma e^{(\tau+\tau_0)}}
    \end{cases},
\end{align*}
then  
\begin{align*}
    u:=
    \begin{cases}
        e^{(-k+1)(\tau+\tau_0)}(Df^k-Bf^{k-1})
        &\text{on }\set{(x,\tau)\in \Sigma\times[0,\infty):\Gamma e^{\tau_0}<f(x)<\gamma_1 e^{(\tau+\tau_0)}}\\
        e^{(\tau+\tau_0)}(a-Ef^{-1})
        &\text{on }\set{(x,\tau)\in \Sigma\times[0,\infty):\gamma_2 e^{(\tau+\tau_0)}<f(x)<\Gamma e^{(\tau+\tau_0)}}
    \end{cases}
\end{align*}
satisfies 
$(\partial_t-Q)u\geq 0$ on its domain if 
\begin{align*}
    &\|h\|_{C^{1,\alpha}_{\rm hom,-1}}\leq \ep<\ep_n,\quad 
    B - D\pr{m(m+\frac{n}{2}) 
    +C_n\sup_\Sigma \pr{f|A|^2+f|\na A|}}
    \geq w>0\\
    &\Gamma \geq \Gamma(n,B,m,w),\quad\gamma_1\leq \gamma_1(n,\Sigma,F,B,D,m,w),\\
    &E - a\pr{C_n\sup_\Sigma \pr{f|A|^2+f|\na A|}}-C_0\sqrt{a}\geq w>0, \text{ and }
    \tau_0\geq \tau_0(n,\Sigma,F,C_0,B,\gamma_2,\ep_n).
\end{align*}
\end{Pro}

This proposition seems very technical, but it enables us to easily control the $C^0$ norm of $h$ using this barrier function.
By combining Proposition~\ref{int1}, Lemma~\ref{int2}, and Lemma~\ref{int3}, we can prove the global $C^2$-improvement theorem following the same argument in \cite{Sto}. 
Note that here we need $C^1$-bounds instead of $C^0$-bounds in the mean curvature flow setting.

\begin{Thm}[For the discussion in Ricci flow settings, see {\cite[Theorem 6.1]{Sto}}]\label{impro}
    Given $\Gamma>0$ and $h$ satisfying \eqref{h-equ} on $\Sigma\times (0,\tau_1)$ and $\tau=-\log(1-t)$, assume 
    \begin{align*}
        \|h\|_{C^{1,\alpha}_{\rm hom,-1}(\Sigma\times [0,\tau_1])}&<\ep,\\
        h(0)\
        &= \eta\pr{\frac{1}{\gamma_0e^{\tau_0}}f(x)} \sum_{i=1}^mp_if_i 
        \text{ for }|\p|\leq \bar{p}e^{-\lambda_*\tau_0},\\
        \|h(\tau)\|_{L_W^2(\Sigma)}&\leq \mu e^{-\lambda_*(\tau+\tau_0)} 
        \text{ for all }\tau\in[0,\tau_1],\\ \supp(\mathcal{Q}_1(\cdot,\tau))
        &\subseteq 
        \set{\frac{\Gamma_0e^{(\tau+\tau_0)}}{4}\leq f\leq \Gamma_0e^{(\tau+\tau_0)}},\\
        |\mathcal{Q}_1(\cdot,\tau)|
        &\leq  C\pr{\frac{1}{\Gamma_0e^{(\tau+\tau_0)}}}^{\frac{1}{2}}, \text{ and}\\
        |\na \mathcal{Q}_1(\cdot,\tau)|
        &\leq  C\pr{\frac{1}{\Gamma_0e^{\tau+\tau_0}}}
    \end{align*}
    for some $\lambda_*>0.$
    Then for any 
    $\delta < \ep$, 
    $\Gamma > \Gamma_0(n,\Sigma,F)$, 
    $\ep<\ep_0(n,\Sigma,F,\Gamma)$, 
    $|\p| < C(n,\Sigma,F,\lambda_*)$, 
    $\mu < \mu_0(n,\Sigma,F,\lambda_*)$, 
    $\gamma < \gamma_0 (n,\Sigma,F,\lambda_*,\delta),$ and 
    $\tau_0 > C(n,\Sigma,F,\lambda_*,\Gamma,\bar{p},\ep,\delta)$, 
    there exist $W=W(n,\Sigma,F,\lambda_*)<\infty$ and $W'=W'(n,\Sigma,F)<\infty$ such that 
    \begin{align*}
       |\na^l h|&\leq W\tilde{r}^{2\lambda_*+1-l}e^{-\lambda_*{(\tau+\tau_0)}},\,\,l\in\{0,1,2\} 
        \text{ for }(x,\tau)\in \Sigma\times[0,\tau_1],\text{ and}\\
        |\na^l h|
        &\leq W'\delta\tilde{r}^{1-l},\,\,l\in\{0,1,2\} 
        \text{ for }(x,\tau)\in \Sigma\times[0,\tau_1].
    \end{align*}
\end{Thm}

\section{\bf Wa\.zewski's box argument} \label{sec:6}
We are going to finish the proof of the main theorem in this section.
For the concerned asymptotically conical shrinker $\Sigma,$ take $\tau_0,\Gamma_0>0$ sufficiently large. 
From Section~\ref{double}, we can construct 
$M = M_{\Sigma,\tau_0,\Gamma_0},$ 
a doubling of a large compact set $\{f(x)\leq 2e^{\tau_0}\Gamma_0\}$ of $\Sigma$. 
Based on Proposition~\ref{prop:eigenfunctions}, we take an $L^2_W$-orthonormal eigenbasis $\{h_i\}_{i\in \mathbb{N}}$ of $L$ on $\Sigma$ and the corresponding eigenvalues $\{\lambda_i\}_{i\in\mathbb{N}}$. 
Take $\lambda_*>0$ and $m$ such that $\lambda_{m}<\lambda_*<\lambda_{m+1}$.

We fix a compactly supported bump function $\eta:(0,\infty)\to [0,1]$ such that $\eta(x)=1$ for $x<\frac{1}{2}$ and $\eta(x)=0$ for $x>1$. 
Take any $\p=(p_1,p_2,\cdots,p_m)\in \RR^m,$ 
and then we can perturb the immersion of $M$ into $\RR^{n+1}$ such that
\begin{align*}
    F_{\p }(x,0)
    = F_M(x) + h_{\p } (x) \N_F(x)
    : = F_M(x)
    +\eta\pr{\frac{1}{\gamma_0e^{\tau_0}}f(x)}
    \sum_{i=1}^mp_if_i\cdot\N_F(x).
\end{align*}
We start the mean curvature flow from $F_{\p}(x,0)$. 
We denote the family of immersions of mean curvature flows by $\{F(x,t)\}_{t\in [0,T(\p )]}$ and transplant it back to the original shrinker as 
\begin{align*}
\hat{F}_{\p }(x,t)
= \eta\pr{\frac{1}{\Gamma_0e^{\tau_0}}f(x)} F_{\p }(x,t)
+\pr{1-\eta\pr{\frac{1}{\Gamma_0e^{\tau_0}}f(x)}}F(x,t),
\end{align*}
where $F(x,t)$ corresponds to the family of immersions of mean curvature flows of the shrinker $\Sigma$.
From the discussion in Section \ref{evolu}, we can take $F_0=\hat{F}_{\p }$. 
Then we get 
\begin{align*}
    \frac{1}{\sqrt{1-t}}\hat{F}_{\p }(\varphi_t(x),t)=F(x)+h_{\p }(x,t)\N_F(x) 
\end{align*}
and
\begin{align*}
  & d\hat{F}_{\p }(\partial_t\varphi_t(x))
  = \frac{\ppair{(\partial_t\hat{F}_{\p })(\varphi_t(x),t)+\frac{1}{2\sqrt{1-t}} F(x),\N_{\hat{F}_{\p }}(\varphi_t(x))}}
  {\ppair{\N_F(x),\N_{\hat{F}_{\p }}(\varphi_t(x))}}
  \N_F(x) - \pr{(\partial_t\hat{F}_{\p })(\varphi_t(x),t)+\frac{1}{2\sqrt{1-t}} F(x)}.
\end{align*}
In particular, we have 
\begin{align*}
    \abs{d\hat{F}_{\p }(\partial_t\varphi_t(x))}
    \leq \frac{\abs{(\partial_t\hat{F}_{\p })(\varphi_t(x),t)+\frac{1}{2\sqrt{1-t}} F(x)}}
    {\abs{\ppair{\N_F(x),\N_{\hat{F}_{\p }}(\varphi_t(x))}}}.
\end{align*}
From \eqref{h-equ}, the function $h=h_{\p }$ satisfies the equation 
\begin{align*}
    \frac{\partial h}{\partial \tau}
    =&Lh+\frac{{\sqrt{1-t}}\la (\partial_t\hat{F}_{\p }-\Delta^{g_0}\hat{F}_{\p })(\varphi_t(x),t),\N_{F_1}(x)\ra}{\la\N_F(x),\N_{F_1}(x)\ra}\notag\\
    &+C_3(F,\na F,h,\na h)* h+C_4(F,\na F,h,\na h)*\na h+C_5(F,\na F,h,\na h)*\na^2 h\\  =&Lh+\mathcal{Q}_1+\mathcal{Q}_2,
\end{align*}
for $\tau=-\log(1-t),$ where
\begin{align*}
    \mathcal{Q}_1=&{\sqrt{1-t}}\la (\partial_t\hat{F}_{\p }-\Delta^{g_0}\hat{F}_{\p })(\varphi_t(x),t),\N_{F_1}(x)\ra,\text{ and}\\ 
    \mathcal{Q}_2=&C_3(F,\na F,h,\na h)*h+C_4(F,\na F,h,\na h)*\na h+C_5(F,\na F,h,\na h)*\na^2 h\notag\\&+{\sqrt{1-t}}
    \ppair{(\partial_t\hat{F}_{\p }-\Delta^{g_0}\hat{F}_{\p })(\varphi_t(x),t),\N_{F_1}(x)} 
    \pr{\frac{1}{{\la\N_F(x),\N_{F_1}(x)\ra}} - 1}.   
\end{align*}
In normal coordinates, the second order term in $\mathcal{Q}_2$ corresponds to 
\begin{align*}
    \frac{\tilde{h}^{ab}\pr{1-|\nabla h|^2} - \pr{|\nabla h|^2+\tilde{h}^{cd}\partial_ch\,\partial_dh} g^{ab}}
    {\pr{1 + \ppair{\N_F,\td \N}}^2}\partial_a\partial_bh\cdot\N_F.
\end{align*}

Given any $h\in L^2_{W}(\Sigma)$, we denote the stable and unstable projections by
\begin{align*}
    h_u &:= \sum_{j=1}^m\pair{h,h_j}_{L^2_W}h_j,\text{ and}\\
    h_s &:= \sum_{j=m+1}^\infty \pair{h,h_j}_{L^2_W} h_j=h-h_u.
\end{align*}
Similar to the Ricci flow case \cite{Sto}, we define the box 
\begin{align*}
    \mathcal{B}[\lambda_*,\mu_u,\mu_s,\ep_0,\ep_1,\ep_2,t_1]
\end{align*}
to be the set which consists of all the functions $h\colon \Sigma\times [0,-\log(1-t_1)]\to \RR$ satisfying the following conditions.\\
(1) $h(\cdot,\tau)$ is compactly supported in $\Sigma$ for any $\tau\in [0,-\log(1-t_1)].$\\
(2) For $\tau\in [0,-\log(1-t_1)],$
\begin{align*}
        \|h_u(\cdot,\tau)\|_{L^2_W}
        &\leq \mu_ue^{-\lambda_*(\tau+\tau_0)},
        \text{ and}\\
        \|h_s(\cdot,\tau)\|_{L^2_W}
        &\leq \mu_se^{-\lambda_*{(\tau+\tau_0)}}.
\end{align*}
(3) For $\tau\in [0,-\log(1-t_1)],$
\begin{align*}
    |h|_{\bar{g}}(x)&\leq \ep_0\tilde{r},\\
    |\na h|_{\bar{g}}&\leq \ep_1, \text{ and}\\
    |\na^2 h|_{\bar{g}}&\leq \ep_2\tilde{r}^{-1}.
\end{align*}
where $\tilde{r}$ is a fixed smooth function such that $\tilde{r}\geq 1$ and $\tilde{r}=|x|$ for $|x|\geq 2$. 
Our main goal here is to prove that for $\gamma_0,|\p|<<1$ and $\Gamma_0, \tau_0>>1$, $h_{\p}$ can only leave $\mathcal{B}$ from the unstable side.

\begin{Pro}[For the discussion in Ricci flow settings, see {\cite[Section 8]{Sto}}]\label{exit}
    Given $\gamma_0,|\p |\leq \bar{p}e^{-\lambda_*\tau_0}<<1$ and $\Gamma_0, \tau_0>>1$, for $h_{\p}$ constructed above, assume $h_{\p }\notin \mathcal{B}[\lambda_*,\mu_u,\mu_s,\ep_0,\ep_1,\ep_2,1].$ 
    If we denote the largest existence time for $F_{\p }(x,\cdot)$ by $T(\p)$ and take the largest time $t(\p)$ such that 
    $h_{\p }\in \mathcal{B} = \mathcal{B}[\lambda_*,\mu_u,\mu_s,\ep_0,\ep_1,\ep_2,t(\p )]$, 
    then we have 
    \begin{align*}
        \|h_{\p ,u}(\cdot,-\log(1-t(\p )))\|_{L^2_W}
        = \mu_u e^{-\lambda_*(\tau_0-\log(1-t(\p )))}. 
    \end{align*}
\end{Pro}

\begin{proof}
From the standard ordinary differential equation and parabolic theories, for $|\p |<<1$, we can get $t(\p )>0$.
From the interior estimate and the pseudolocality in \cite{EH, INS}, on 
$M\setminus \set{f\leq \frac{\Gamma_0e^{\tau_0}}{2}}$, 
$F_{\p }(x,t)$ satisfies 
$|\na^mA|\leq C_m(\frac{1}{\Gamma_0e^{\tau_0}})^{\frac{m+1}{2}}$ 
for $t\in [0,2]$. 
Thus, if $h_{\p }\in \mathcal{B}$, we have  \begin{align*}
    |(\partial_t\hat{F}_{\p }-\Delta^{g_0}\hat{F}_{\p })|(x,t)
    &\leq  C\pr{\frac{1}{\Gamma_0e^{\tau_0}}}^{\frac{1}{2}},\text{ and}\\
    |\na (\partial_t\hat{F}_{\p }-\Delta^{g_0}\hat{F}_{\p })|(x,t)
    &\leq  C\pr{\frac{1}{\Gamma_0e^{\tau_0}}}.
\end{align*} 
Hence, if we denote $\tau=-\log(1-t)$, we have
\begin{align*}
    \supp(\mathcal{Q}_1(\cdot,\tau))
    &=\supp\pr{{\sqrt{1-t}} \ppair{(\partial_t\hat{F}_{\p}-\Delta^{g_0}\hat{F}_{\p})(\varphi_t(\cdot),t),\N_{F_1}(\cdot)}}
    \subseteq
    \set{\frac{\Gamma_0e^{(\tau+\tau_0)}}{2}\leq f\leq \Gamma_0e^{(\tau+\tau_0)}},\\
    |\mathcal{Q}_1(\cdot,t)|
    &\leq  C\pr{\frac{1}{\Gamma_0e^{\tau+\tau_0}}}^{\frac{1}{2}},
    \text{ and }
    |\na \mathcal{Q}_1(\cdot,t)|\leq  C\pr{\frac{1}{\Gamma_0e^{\tau+\tau_0}}}.
\end{align*}
From Theorem \ref{impro}, given $\tau_0>>1,$ we can improve the $C^2$ bounds such that \begin{align*}
    \|h\|_{C^2(\Sigma\times[0,\tau_1])}
    &\leq W'\delta
    < \frac{1}{2}\min\{\ep_1,\ep_2,\ep_3\},\text{ and}\\
    \|h\|_{C^2(\Sigma\times[0,\tau_1])}
    &\leq W\pr{\frac{|F|^2}{4}}^{\lambda_*} e^{-\lambda_*(\tau+\tau_0)}.
\end{align*}
From the equation, given $|h|_{C^2}<\ep_0(n)$, \begin{align*}
    |\mathcal{Q}_2|\leq C \pr{|h|+|\na h|+|\na^2 h|}^2 \leq C
    \min \set{W\pr{\frac{|F|^2}{4}}^{\lambda_*} e^{-\lambda_*(\tau+\tau_0)},\ep_1,\ep_2,\ep_3}.
\end{align*}
Thus, given $\tau_0>>1$, 
\begin{align}\label{error}
    \|\mathcal{Q}_1\|_{L^2_W} + \|\mathcal{Q}_2\|_{L^2_W}
    \leq CW^2e^{-2\lambda_*(\tau+\tau_0)}\leq Ce^{-2\lambda_*(\tau+\tau_0)}.
\end{align} 
From the definition of $h_{\p }$, the polynomial decay of eigenfunctions (Theorem~\ref{thm:eigen-decay}), and the exponential decay of the weight,
for $|\p |\leq\bar{p}e^{-\lambda_*\tau_0}$, we have
\begin{align}\label{initial}
    \abs{\la h_{\p }(\cdot,0),h_j\ra_{L^2_W} -p_j}
    &\leq Ce^{-\frac{\gamma_0}{100}e^{\tau_0}} \text{ for }j\leq m,\text{ and}\\
    \|h_{\p ,s}\|_{L^2_W}
    &\leq Ce^{-\frac{\gamma_0}{100}e^{\tau_0}}.\nonumber
\end{align}
From \eqref{error} and the equation $\partial_\tau h=Lh+\mathcal{Q}_1+\mathcal{Q}_2$, 
\begin{align}\label{unstable}
    \frac{d}{d\tau}\|e^{\lambda_*(\tau+\tau_0)} h_{\p ,u}\|_{L^2_W}
    &\geq (\lambda_*-\lambda_m)\|e^{\lambda_*(\tau+\tau_0)} h_{\p ,u}\|_{L^2_W} -Ce^{-\lambda_*\tau_0},\text{ and}\\
     \frac{d}{d\tau}\|e^{\lambda_*(\tau+\tau_0)} h_{\p ,s}\|_{L^2_W}
     &\leq (\lambda_*-\lambda_{m+1})\|e^{\lambda_*(\tau+\tau_0)} h_{\p ,u}\|_{L^2_W} +Ce^{-\lambda_*\tau_0}. \nonumber
\end{align}
In particular, given 
$\tau_0>C(n,\mu_s,\mu_u),$ 
\begin{align*}
    \|h_{\p ,s}(\cdot,-\log(1-t(\p )))\|_{L^2_W}
    < \mu_se^{-\lambda_*(\tau_0-\log(1-t(\p )))}, 
\end{align*}
and thus, from the definition of the box, we obtain
\begin{align*}
        \|h_{\p ,u}(\cdot,-\log(1-t(\p )))\|_{L^2_W}= \mu_ue^{-\lambda_*(\tau_0-\log(1-t(\p )))} .
    \end{align*}
Moreover, from \eqref{unstable}, for $\tau>-\log(1-t(\p ))$,
\begin{align*}
    \|h_{\p ,u}(\cdot,\tau)\|_{L^2_W}> \mu_ue^{-\lambda_*(\tau+\tau_0)}.
\end{align*}
This completes the proof.
\end{proof}

Next, we prove that we can choose a suitable perturbation such that the flow remains in the box until time $t=1$.
This is based on a topological argument.

\begin{Pro}[For the discussion in Ricci flow settings, see {\cite[Section 8]{Sto}}]\label{perturbation}
    Given $\gamma_0,$ $\bar{p}e^{-\lambda_*\tau_0}<<1$ and $\Gamma_0, \tau_0>>1$,  
    there exists $\p \in \overline{B_{\bar{p}e^{-\lambda_*\tau_0}}}$ such that 
    \begin{align}
        h_{\p }\in \mathcal{B}[\lambda_*,\mu_u,\mu_s,\ep_0,\ep_1,\ep_2,1].
    \end{align}
\end{Pro}

\begin{proof}
As in the Ricci flow case, we define $\mathcal{F}:\overline{B_{\bar{p}e^{-\lambda_*\tau_0}}}\to \RR^m$ by
\begin{align*}
    \mathcal{F}(\p) 
    = \pr{\ppair{h_{\p }(t(\p )),h_i}_{L^2_W}}_{i=1,2,\cdots,m}.
\end{align*}
If we take $\mu_u<\frac{\bar{p}}{2}$, from \eqref{initial},
the map
$$\mathcal{F}|_{\overline{B_{\bar{p}e^{-\lambda_*\tau_0}}} \setminus B_{2\mu_ue^{-\lambda_*\tau_0}}}
\colon \overline{B_{\bar{p}e^{-\lambda_*\tau_0}}}\setminus B_{2\mu_ue^{-\lambda_*\tau_0}}\to \RR^m\setminus\{0\}$$ 
is homotopic to the identity through 
$\mathcal{F}_s := s\mathcal{F}+(1-s){\rm Id}
\colon \overline{B_{\bar{p}e^{-\lambda_*\tau_0}}}\setminus B_{2\mu_ue^{-\lambda_*\tau_0}}\to \RR^m\setminus\{0\}$.

We claim that this shows the Proposition. Otherwise, if the proposition is false, then
\begin{align*}
        h_{\p }\notin \mathcal{B}[\lambda_*,\mu_u,\mu_s,\ep_0,\ep_1,\ep_2,1]
\end{align*}
 for all $\p \in \overline{B_{\bar{p}e^{-\lambda_*\tau_0}}}$. Thus Proposition~\ref{exit} implies that 
$|\mathcal{F}(\p)|>0$ for all $\p \in \overline{B_{\bar{p}e^{-\lambda_*\tau_0}}}$. 
Note that $\mathcal{F}|_{\overline{B_{\bar{p}e^{-\lambda_*\tau_0}}}}$ is null-homotopic to a constant map by shrinking $\overline{B_{\bar{p}e^{-\lambda_*\tau_0}}}$ to a point and compose it with $\mathcal{F}$. 
Here, we have that the whole homotopy lies in $\RR^m\setminus\{0\}$ as $|\mathcal{F}(\p)|>0$ for all $\p \in \overline{B_{\bar{p}e^{-\lambda_*\tau_0}}}$. 
This shows that, in particular, $$\mathcal{F}|_{\overline{B_{\bar{p}e^{-\lambda_*\tau_0}}}\setminus B_{2\mu_ue^{-\lambda_*\tau_0}}}\colon \overline{B_{\bar{p}e^{-\lambda_*\tau_0}}}\setminus B_{2\mu_ue^{-\lambda_*\tau_0}}\to \RR^m\setminus\{0\}$$ 
is null-homotopic.

Hence, we derive a contraction.
This proves that there exists $\p \in \overline{B_{\bar{p}e^{-\lambda_*\tau_0}}}$ such that 
\begin{align}
        h_{\p }\in \mathcal{B}[\lambda_*,\mu_u,\mu_s,\ep_0,\ep_1,\ep_2,1],
\end{align}
completing the proof.
\end{proof}

We can prove our main result (Theorem~\ref{main}).

\begin{Thm}
    Given an asymptotically conical shrinker $\Sigma^n\sbst \RR^{n+1}$, there exists a closed embedded hypersurface $M\sbst \RR^{n+1}$ such that the mean curvature flow $M_t$ starting from $M$ develops a type I singularity at time $1$ at the origin modeled on $\Sigma$.
\end{Thm}

\begin{proof}
    For any $\p,$ if we blow up the mean curvature flow starting from the corresponding perturbation at $({\bf 0},1)$, the family of rescaled hypersurfaces is the graph of $h_{\p }$ over $\Sigma$. 
    From Proposition~\ref{perturbation}, there exists a vector $\p$ such that $ h_{\p }\in \mathcal{B}[\lambda_*,\mu_u,\mu_s,\ep_0,\ep_1,\ep_2,1].$ 
    Based on Theorem~\ref{impro} and Proposition~\ref{hHolder}, $h_{\p}$ converges to $0$ in $C^\infty_{\rm loc}(\Sigma)$. 
    The global $C^2$ bound of $h_{\p}$ guarantees that the singularity is of type I. 
    This finishes the proof. 
\end{proof}

Using Theorem~\ref{main} and the results from \cite{AIV,CDS}, we can prove Corollary~\ref{cor:fatten}.

\begin{proof}[Proof of Corollary~\ref{cor:fatten}]
    Angenent--Ilmanen--Velázquez constructed $n$-dimensional asymptotically conical shrinkers such that level set flows starting from their asymptotic cones fatten for $n\in\set{3,4,5,6}$
    in \cite[Theorem~1]{AIV}.\footnote{The proof of Theorem~1 in \cite{AIV} was provided in Section~3 (for proving Theorem~5) and Section~5.1 there.
    }
    Take such a shrinker $\Sigma$ and let $\mathcal C$ be its asymptotic cone.
    Theorem~\ref{main} implies that there exists a smooth closed hypersurface $M\sbst\bb R^{n+1}$ such that the level set flow $M_t$ starting from $M$ develops a singularity at $({\bf 0},0)\in\bb R^{n+1}\times\bb R$ whose tangent flow is modeled on $\Sigma.$
    Thus, $M_0$ is a hypersurface with a singularity at the origin modeled on the cone $\mathcal C$ (cf. \cite[Corollary 1.2]{CS}).
    Since the level set flow starting from $\mathcal C$ fattens, the main result in \cite{CDS} implies that the level set flow $M_t$ fattens for positive time.
\end{proof}


\begin{thebibliography}{9999999}
\bib{AA}{article}{
   author={Allard, W. K.},
   author={Almgren, F. J., Jr.},
   title={On the radial behavior of minimal surfaces and the uniqueness of
   their tangent cones},
   journal={Ann. of Math. (2)},
   volume={113},
   date={1981},
   number={2},
   pages={215--265},
   issn={0003-486X},
   review={\MR{607893}},
}


\bib{AIC}{article}{
   author={Angenent, S.},
   author={Chopp, D. L.},
   author={Ilmanen, T.},
   title={A computed example of nonuniqueness of mean curvature flow in
   $\bold R^3$},
   journal={Comm. Partial Differential Equations},
   volume={20},
   date={1995},
   number={11-12},
   pages={1937--1958},
   issn={0360-5302},
   review={\MR{1361726}},
}

\bib{AIV}{article}{
	author={Angenent, S.},
	author={Ilmanen, T.},
    author={Vel\'azquez, J. J. L.}
	title={Fattening from smooth initial data in mean curvature flow},
	note={Preprint available at \url{http://people.math.wisc.edu/~angenent/preprints/unfinished/fattening.pdf}},
	date={2017}
}

\bib{App}{article}{
   author={Appleton, A.},
   title={Scalar curvature rigidity and Ricci DeTurck flow on perturbations
   of Euclidean space},
   journal={Calc. Var. Partial Differential Equations},
   volume={57},
   date={2018},
   number={5},
   pages={Paper No. 132, 23},
   issn={0944-2669},
   review={\MR{3844513}},
}

 \bib{Bam}{article}{
   author={Bamler, R. H.},
   title={Stability of hyperbolic manifolds with cusps under Ricci flow},
   journal={Adv. Math.},
   volume={263},
   date={2014},
   pages={412--467},
   issn={0001-8708},
   review={\MR{3239144}},
}

\bib{BK23}{article}{
	author={Bamler, R. H.},
	author={Kleiner, B.},
	title={On the multiplicity one conjecture for mean curvature flows of surfaces
	},
	note={Preprint available at \url{https://arxiv.org/abs/2312.02106}},
	date={2023}
}
\bib{Bra69}{article}{
   author={Brandt, A.},
   title={Interior Schauder estimates for parabolic differential- (or
   difference-) equations via the maximum principle},
   journal={Israel J. Math.},
   volume={7},
   date={1969},
   pages={254--262},
   issn={0021-2172},
   review={\MR{0249803}},
   doi={10.1007/BF02787619},
}
\bib{BW17}{article}{
   author={Bernstein, J.},
   author={Wang, L.},
   title={A topological property of asymptotically conical self-shrinkers of
   small entropy},
   journal={Duke Math. J.},
   volume={166},
   date={2017},
   number={3},
   pages={403--435},
   issn={0012-7094},
   review={\MR{3606722}},
}
	
	\bib{B78}{book}{
		author={Brakke, K. A.},
		title={The motion of a surface by its mean curvature},
		series={Mathematical Notes},
		volume={20},
		publisher={Princeton University Press, Princeton, N.J.},
		date={1978},
		pages={i+252},
		isbn={0-691-08204-9},
		review={\MR{485012}},
	}


	\bib{Brandt}{article}{
		author={Brandt, A.},
		title={Interior Schauder estimates for parabolic differential- (or
			difference-) equations via the maximum principle},
		journal={Israel J. Math.},
		volume={7},
		date={1969},
		pages={254--262},
		issn={0021-2172},
		review={\MR{249803}},
	}
	
\bib{CCMS}{article}{
	author={Chodosh, O.},
	author={Choi, K.},
    author={Mantoulidis, C.},
    author={Schulze, F.}
	title={Mean curvature flow with generic initial data},
	note={Preprint available at \url{https://arxiv.org/abs/2003.14344}},
	date={2020}
}

\bib{CCS}{article}{
	author={Chodosh, O.},
	author={Choi, K.},
    author={Schulze, F.}
	title={Mean curvature flow with generic initial data II},
	note={Preprint available at \url{https://arxiv.org/abs/2302.08409}},
	date={2023}
}

\bib{CDS}{article}{
	author={Chodosh, O.},
	author={Daniels-Holgate, J. M.},
    author={Schulze, F.}
	title={Mean Curvature Flow from Conical Singularities},
	note={Preprint available at \url{https://arxiv.org/abs/2312.00759}},
	date={2023}
}

    \bib{CS}{article}{
		author={Chodosh, O.},
		author={Schulze, F.},
		title={Uniqueness of asymptotically conical tangent flows},
		journal={Duke Math. J.},
		volume={170},
		date={2021},
		number={16},
		pages={3601--3657},
		issn={0012-7094},
		review={\MR{4332673}},
	}

    \bib{CSSZ}{article}{
	author={Choi, K.},
    author={Seo, D.-H.},
    author={Su, W.-B.},
    author={Zhao, K.-W.},
	title={Uniqueness of tangent flows at infinity for finite-entropy shortening curves
},
	note={Preprint available at \url{https://arxiv.org/abs/2405.10664}},
	date={2024}
}


	
	\bib{CM12}{article}{
		author={Colding, T. H.},
		author={Minicozzi, W. P., II},
		title={Generic mean curvature flow I: generic singularities},
		journal={Ann. of Math. (2)},
		volume={175},
		date={2012},
		number={2},
		pages={755--833},
		issn={0003-486X},
		review={\MR{2993752}},
	}

    
\bib{CM14}{article}{
   author={Colding, T. H.},
   author={Minicozzi, W. P., II},
   title={On uniqueness of tangent cones for Einstein manifolds},
   journal={Invent. Math.},
   volume={196},
   date={2014},
   number={3},
   pages={515--588},
   issn={0020-9910},
   review={\MR{3211041}},
}
	
	\bib{CM15}{article}{
		author={Colding, T. H.},
		author={Minicozzi, W. P., II},
		title={Uniqueness of blowups and \L ojasiewicz inequalities},
		journal={Ann. of Math. (2)},
		volume={182},
		date={2015},
		number={1},
		pages={221--285},
		issn={0003-486X},
		review={\MR{3374960}},
	}




\bib{CM20}{article}{
		author={Colding, T. H.},
		author={Minicozzi, W. P., II},
   title={Complexity of parabolic systems},
   journal={Publ. Math. Inst. Hautes \'{E}tudes Sci.},
   volume={132},
   date={2020},
   pages={83--135},
   issn={0073-8301},
   review={\MR{4179832}},
}

\bib{CM21}{article}{
		author={Colding, T. H.},
		author={Minicozzi, W. P., II},
		title={Singularities of Ricci flow and diffeomorphisms},
		note={Preprint available at \url{https://arxiv.org/abs/2109.06240}},
		date={2021}
	}

 \bib{CM19}{article}{
		author={Colding, T. H.},
		author={Minicozzi, W. P., II},
   title={Regularity of elliptic and parabolic systems},
   language={English, with English and French summaries},
   journal={Ann. Sci. \'{E}c. Norm. Sup\'{e}r. (4)},
   volume={56},
   date={2023},
   number={6},
   pages={1883--1921},
   issn={0012-9593},
   review={\MR{4706576}},
}
 
\bib{De}{article}{
   author={De Giorgi, E.},
   title={Some conjectures on flow by mean curvature},
   conference={
      title={Methods of Real Analysis and Partial Differential Equations},
      address={Liguori, Napoli},
      date={1990},
   },
}

 \bib{EH}{article}{
   author={Ecker, K.},
   author={Huisken, G.},
   title={Interior estimates for hypersurfaces moving by mean curvature},
   journal={Invent. Math.},
   volume={105},
   date={1991},
   number={3},
   pages={547--569},
   issn={0020-9910},
   review={\MR{1117150}},
}

\bib{ES}{article}{
		author={Evans, L. C.},
		author={Spruck, J.},
		title={Motion of level sets by mean curvature. I},
		journal={J. Differential Geom.},
		volume={33},
		date={1991},
		number={3},
		pages={635--681},
		issn={0022-040X},
		review={\MR{1100206}},
	}

 \bib{GS18}{article}{
   author={Guo, S.-H.},
   author={Sesum, N.},
   title={Analysis of Vel\'{a}zquez's solution to the mean curvature flow with a
   type II singularity},
   journal={Comm. Partial Differential Equations},
   volume={43},
   date={2018},
   number={2},
   pages={185--285},
   issn={0360-5302},
   review={\MR{3777874}},
}

    \bib{H97}{article}{
   author={Hardt, R. M.},
   title={Singularities of harmonic maps},
   journal={Bull. Amer. Math. Soc. (N.S.)},
   volume={34},
   date={1997},
   number={1},
   pages={15--34},
   issn={0273-0979},
   review={\MR{1397098}},
}
 
	
	\bib{H90}{article}{
		author={Huisken, G.},
		title={Asymptotic behavior for singularities of the mean curvature flow},
		journal={J. Differential Geom.},
		volume={31},
		date={1990},
		number={1},
		pages={285--299},
		issn={0022-040X},
		review={\MR{1030675}},
	}

\bib{I94}{article}{
   author={Ilmanen, T.},
   title={Elliptic regularization and partial regularity for motion by mean
   curvature},
   journal={Mem. Amer. Math. Soc.},
   volume={108},
   date={1994},
   number={520},
   pages={x+90},
   issn={0065-9266},
   review={\MR{1196160}},
}
 
	\bib{I95}{article}{
		author={Ilmanen, T.},
		title={Singularities of Mean Curvature Flow of Surfaces},
		note={Preprint available at \url{http://www.math.ethz.ch/~ilmanen/papers/pub.html}},
		date={1995}
	}

\bib{INS}{article}{
   author={Ilmanen, T.},
   author={Neves, A.},
   author={Schulze, F.},
   title={On short time existence for the planar network flow},
   journal={J. Differential Geom.},
   volume={111},
   date={2019},
   number={1},
   pages={39--89},
   issn={0022-040X},
   review={\MR{3909904}},
}

\bib{IW}{article}{
   author={Ilmanen, T.},
   author={White, B.},
		title={Fattening in mean curvature flow
},
		note={Preprint available at \url{https://arxiv.org/abs/2406.18703}},
		date={2024}
	}


	\bib{KKM}{article}{
		author={Kapouleas, N.},
		author={Kleene, S. J.},
		author={M\o ller, N. M.},
		title={Mean curvature self-shrinkers of high genus: non-compact examples},
		journal={J. Reine Angew. Math.},
		volume={739},
		date={2018},
		pages={1--39},
		issn={0075-4102},
		review={\MR{3808256}},
	}

\bib{K24}{article}{
		author={Ketover, D.},
		title={Self-shrinkers whose asymptotic cones fatten
},
		note={Preprint available at \url{https://arxiv.org/abs/2407.01240}},
		date={2024}
	}

\bib{Kh}{article}{
   author={Khan, I.},
   title={Uniqueness of asymptotically conical higher codimension
   self-shrinkers and self-expanders},
   journal={J. Reine Angew. Math.},
   volume={805},
   date={2023},
   pages={55--99},
   issn={0075-4102},
   review={\MR{4669035}},
}

 \bib{Kry}{book}{
   author={Krylov, N. V.},
   title={Lectures on elliptic and parabolic equations in H\"{o}lder spaces},
   series={Graduate Studies in Mathematics},
   volume={12},
   publisher={American Mathematical Society, Providence, RI},
   date={1996},
   pages={xii+164},
   isbn={0-8218-0569-X},
   review={\MR{1406091}},
}

\bib{LZ}{article}{
   author={Lee, T.-K.},
   author={Zhao, X.},
   title={Uniqueness of conical singularities for mean curvature flows},
   journal={J. Funct. Anal.},
   volume={286},
   date={2024},
   number={1},
   pages={Paper No. 110200, 23},
   issn={0022-1236},
   review={\MR{4656691}},
}

\bib{Liu}{article}{
		author={Liu, Z.},
		title={Blow Up of Compact Mean Curvature Flow Solutions with Bounded Mean Curvature},
		note={Preprint available at \url{https://arxiv.org/abs/2403.16515}},
		date={2024},
	}
	
	\bib{LSS}{article}{
		author={Lotay, J. D.},
		author={Schulze, F.},
		author={Sz\'ekelyhidi, G.},
		title={Neck pinches along the Lagrangian mean curvature flow of surfaces},
		note={Preprint available at \url{https://arxiv.org/abs/2208.11054}},
		date={2022}
	}

\bib{N14}{article}{
   author={Nguyen, X. H.},
   title={Construction of complete embedded self-similar surfaces under mean
   curvature flow, Part III},
   journal={Duke Math. J.},
   volume={163},
   date={2014},
   number={11},
   pages={2023--2056},
   issn={0012-7094},
   review={\MR{3263027}},
}

	\bib{S14}{article}{
		author={Schulze, F.},
		title={Uniqueness of compact tangent flows in mean curvature flow},
		journal={J. Reine Angew. Math.},
		volume={690},
		date={2014},
		pages={163--172},
		issn={0075-4102},
		review={\MR{3200339}},
	}

	\bib{S83}{article}{
		author={Simon, L.},
		title={Asymptotics for a class of nonlinear evolution equations, with
			applications to geometric problems},
		journal={Ann. of Math. (2)},
		volume={118},
		date={1983},
		number={3},
		pages={525--571},
		issn={0003-486X},
		review={\MR{727703}},
	}
	


    \bib{S08}{article}{
   author={Simon, L.},
   title={A general asymptotic decay lemma for elliptic problems},
   conference={
      title={Handbook of geometric analysis. No. 1},
   },
   book={
      series={Adv. Lect. Math. (ALM)},
      volume={7},
      publisher={Int. Press, Somerville, MA},
   },
   date={2008},
   pages={381--411},
   review={\MR{2483370}},
}

\bib{Sto}{article}{
  title={Closed Ricci flows with singularities modeled on asymptotically conical shrinkers},
  author={Stolarski, M.},
  journal={Preprint available at \url{https://arxiv.org/abs/2202.03386}},
  year={2022}
}


\bib{S23V}{article}{
   author={Stolarski, M.},
   title={Existence of mean curvature flow singularities with bounded mean
   curvature},
   journal={Duke Math. J.},
   volume={172},
   date={2023},
   number={7},
   pages={1235--1292},
   issn={0012-7094},
   review={\MR{4583651}},
}

\bib{S23}{article}{
		author={Stolarski, M.},
		title={On the Structure of Singularities of Weak Mean Curvature Flows with Mean Curvature Bounds
		},
		note={Preprint available at \url{https://arxiv.org/abs/2311.16262}},
		date={2023}
	}

    \bib{SX1}{article}{
		author={Sun, A.},
		author={Xue, J.},
		title={Initial perturbation of the mean curvature flow for closed limit shrinker
		},
		note={Preprint available at \url{https://arxiv.org/abs/2104.03101}},
		date={2021}
	}

    \bib{SX2}{article}{
		author={Sun, A.},
		author={Xue, J.},
		title={Initial perturbation of the mean curvature flow for asymptotical conical limit shrinker
		},
		note={Preprint available at \url{https://arxiv.org/abs/2107.05066}},
		date={2021}
	}

 
\bib{V94}{article}{
   author={Vel\'{a}zquez, J. J. L.},
   title={Curvature blow-up in perturbations of minimal cones evolving by
   mean curvature flow},
   journal={Ann. Scuola Norm. Sup. Pisa Cl. Sci. (4)},
   volume={21},
   date={1994},
   number={4},
   pages={595--628},
   issn={0391-173X},
   review={\MR{1318773}},
}


	\bib{W14}{article}{
		author={Wang, L.},
		title={Uniqueness of self-similar shrinkers with asymptotically conical
			ends},
		journal={J. Amer. Math. Soc.},
		volume={27},
		date={2014},
		number={3},
		pages={613--638},
		issn={0894-0347},
		review={\MR{3194490}},
	}

    \bib{W16}{article}{
		author={Wang, L.},
		title={Asymptotic structure of self-shrinkers
		},
		note={Preprint available at \url{https://arxiv.org/abs/1610.04904}},
		date={2016}
	}
	
    \bib{W94}{article}{
		author={White, B.},
		title={Partial regularity of mean-convex hypersurfaces flowing by mean
			curvature},
		journal={Internat. Math. Res. Notices},
		date={1994},
		number={4},
		pages={185-192},
			issn={1073-7928},
			review={\MR{1266114}},
	}

 \bib{W02}{article}{
   author={White, B.},
   title={Evolution of curves and surfaces by mean curvature},
   conference={
      title={Proceedings of the International Congress of Mathematicians,
      Vol. I},
      address={Beijing},
      date={2002},
   },
   book={
      publisher={Higher Ed. Press, Beijing},
   },
   date={2002},
   pages={525--538},
   review={\MR{1989203}},
}

	
	\bib{Z20}{article}{
		author={Zhu, J. J.},
		title={\L ojasiewicz inequalities, uniqueness and rigidity for cylindrical self-shrinkers},
		note={Preprint available at \url{https://arxiv.org/abs/2011.01633}},
		date={2020}
	}
\end{thebibliography}
\end{document}